\documentclass[reqno,12pt]{amsart} 
\usepackage{mathrsfs}
\usepackage[margin=1in]{geometry}
\usepackage{amssymb,amsmath,amsfonts,amsthm}
\usepackage{latexsym}
\usepackage{amscd}
\usepackage{texdraw}
\usepackage{color}
\usepackage{cite}
\usepackage[centertags]{amsmath}
\usepackage[colorlinks=true,urlcolor=blue,linkcolor=red,citecolor=red]{hyperref}

\DeclareMathOperator{\supp}{supp}
\newlength{\defbaselineskip}
\newcommand{\setlinespacing}[1]%
           {\setlength{\baselineskip}{#1 \defbaselineskip}}

\newcommand{\R}{\mathbb R}
\newcommand{\rd}{\R^d}
\newcommand{\C}{\mathbb C}
\newcommand{\Z}{\mathbb Z}
\newcommand{\T}{\mathbb T}
\newcommand{\N}{\mathbb N}

\newcommand{\zd}{\Z^d}
\newcommand{\cX}{\mathcal{X}}

\newcommand{\be}{\begin{equation}}
\newcommand{\ee}{\end{equation}}
\newcommand{\bea}{\begin{eqnarray}}
\newcommand{\eea}{\end{eqnarray}}
\newcommand{\Bea}{\begin{eqnarray*}}
\newcommand{\Eea}{\end{eqnarray*}}
\newcommand{\bt}{\begin{Theorem}}
\newcommand{\et}{\end{Theorem}}

\theoremstyle{plain}
\newtheorem{theorem}{Theorem}[section]
\newtheorem{lemma}{Lemma}[section]
\newtheorem{proposition}{Proposition}[section]
\newtheorem{corollary}{Corollary}[section]
\newtheorem{definition}{Definition}[section]
\newtheorem{remark}{Remark}[section]

\numberwithin{equation}{section}

\allowdisplaybreaks

\begin{document}
\title[On the weighted Wiener-L\'evy theorem]{On the weighted Wiener-L\'evy theorem: analogue on Euclidean space, strong converse on LCA group\\ and applications to modulation spaces}

\author[D. G. Bhimani]{Divyang G. Bhimani}
\address{Department of Mathematics, Indian Institute of Science Education and Research, Pune, India}
\email{divyang.bhimani@iiserpune.ac.in}

\author[K. B. Solanki]{Karishman B. Solanki}
\address{Department of Mathematics, Indian Institute of Technology Ropar, Punjab, India}
\email{karishsolanki002@gmail.com, staff.karishman.solanki@iitrpr.ac.in}

\thanks{}



\date{}

\dedicatory{}

\commby{}

\begin{abstract}
We consider  the space (weighted Fourier algebra) of  Banach algebra valued functions $A^q_{\omega}(\Gamma,\cX),$ which consists of  all Fourier transforms of functions in $L^q_\omega(G,\cX)$. Here  $\omega$ is a  Beurling-Domar type weight  on a discrete abelian group $G$, $\Gamma$ is the dual of $G$, and  $\cX$ is a unital commutative Banach algebra. We shall prove a strong converse of the Wiener-L\'evy theorem in vector valued weighted setting. Specifically, we proved that  if $F$ is a $\cX-$valued function defined on $\mathbb{C}$ such that the composition $F\circ f:\Gamma\to\cX$ is in $A^q(\Gamma,\cX)$ ($1\leq q <2$) for all $f\in A^1_\omega(\Gamma,\mathbb{C})$, then $F$ must be real analytic on $\mathbb R^2$. Here the range of $q$ is sharp. Further, its multivariate analogue and analogue for locally compact abelian $G$ are also established. 
This is  the first result  which generalizes the classic theorems of  Helson, Kahane, Katznelson and Rudin  \cite{helson,rud} in the presence of proposed  weight.  On the other hand, we established the analogue of Wiener-L\'evy theorem in  Euclidean  Fourier algebra $A_{\omega}^q(\mathbb R^d)$ for a weight $\omega$ of regular growth.

As an application, we establish similar  results for weighted   modulation,  Wiener amalgam  and  Fourier amalgam spaces. This complements the work of Bhimani-Ratnakumar \cite{bhimani2016functions} and Feichtinger-Kobayashi-Sato \cite{HGWL1, HGWL2}; and  enables us to shed light on nonlinearity while understanding the dynamics of dispersive PDE in these spaces.

\end{abstract}

\maketitle
\tableofcontents

\section{Introduction}
\subsection{Background and motivation}
Let $A(\mathbb{T})$ be the collection of all complex valued continuous function on the unit circle $\mathbb{T}$ of the complex plane $\mathbb{C}$ having absolutely convergent Fourier series. The classical Wiener's theorem \cite{NW} states that if $f\in A(\mathbb{T})$ is nowhere zero, then $\frac{1}{f}\in A(\mathbb{T})$.  L\'evy \cite{Le} generalized Wiener's theorem and showed that if $F$ is an analytic function in some neighborhood containing the range of $f\in A(\mathbb{T})$, then $F\circ f\in A(\mathbb{T})$. In fact, the converse is also true, and its first appearance dates back to  Katznelson \cite{Katz1959}. Given  a function  $F:E\to\mathbb{C}$, where $E\subset \mathbb{C}$, we  define a composition operator $$T_F(f)=F\circ f$$ for $f\in A(\mathbb T)$ whose range is contained in $E$. We restate combined classical Wiener-L\'evy and Katznelson theorem as follows:
\begin{theorem}[Wiener-L\'evy-Katznelson]
    A composition operator $T_F$ maps $A(\mathbb T)$ into $A(\mathbb T)$ if and only if $F$ is real analytic in $E.$
\end{theorem}
In analysis, \textit{weights} are  frequently  employed   as they are beneficial in many situations (e.g. to quantify the growth and decay conditions).  In fact, to study the behavior of functions around a certain point (and asymptotic behavior), weights are used to ignore their oscillations at infinity.  This has many applications in various fields of analysis such as signal theory, frame theory, sampling theory, Gabor analysis, etc.; for instance, refer to \cite{Badi,Ba,kb3,Gro,Gr, Bal25,FeiWeight1979, GroWeight2007, Ku} and references therein. 

It would be interesting  to inquire whether Wiener-L\'evy-Katznelson theorem can be extended in the presence of weights. Wiener-L\'evy theorem for weighted spaces has been studied extensively, see for example \cite{Kim, Badi, BhDeDa, Ba, Bhatt, kb2, kb, kb3, ki}.  Another motivation for us comes from the fact that composition operators are simple examples of non-linear mappings (non linearity). This knowledge plays a  crucial role  in studying nonlinear PDEs. See \cite{bhimani2016functions,SugimotoNOPJFA,Sugimoto2011, db, BhimaniITSP} and \cite[Chapter 5]{kassob}.

On the other hand, in analysis, vector valued cases are of importance as many times $B(\mathcal{Y})$-valued functions arises for investigation for some Banach or Hilbert space $\mathcal{Y}$, particularly in operator theory. We note that the vector valued analogues of the Wiener's theorem were first studied by Bochner Phillips in \cite{Bo} and more general results with weights and $p-$th power has been studied in \cite{kb,kb2}. The $B(\mathcal{Y})$-valued different analogues has also been investigated for different types of operators. We again refer the readers to \cite{Badi, Ba, kb2, kb, kb3} and references therein for more details. And this motivated us to consider the vector valued cases here.

In order to state main results of this paper, we briefly set notations. Throughout the paper, $G$ is a locally compact (noncompact) abelian group and $\Gamma$ is its (Pontryagin)\footnote{In fact, $\Gamma$ is  the group of continuous group homomorphisms from  $G$ to the circle group $\mathbb T$} dual group. The binary operation on $G$ will be addition, and that on $\Gamma$ will be multiplication. In order to avoid trivialities, we take $G$ and hence $\Gamma$ to be infinite. 

\begin{definition} \
    \begin{enumerate}
        \item A \textbf{\emph{weight}}  on $G$ is a measurable map $\omega:G\to[1,\infty)$ which satisfies the submultiplicative condition: $$\omega(x+y)\leq\omega(x)\omega(y) \quad \text{for   all} \ \ x,y\in G. $$ 
        \item A weight $\omega$ on $G$ is \textbf{\emph{admissible}} if it satisfies the \emph{Beurling-Domar (BD) condition} 
\begin{equation} \label{BDcond}
    \sum_{n\in\mathbb{N}} \frac{\log \omega(nx)}{n^2}<\infty \quad \text{for   all} \ \ x\in G.
\end{equation}
\item  For $1<q<\infty$, a weight $\omega$ on $G$ is a \textbf{\emph{$q-$algebra weight}} if it satisfies the following conditions
\begin{enumerate} 
    \item[(a)] subconvolutive: $\omega^{-q'}\ast\omega^{-q'}\leq \omega^{-q'}$, and 
    \item[(b)] $\displaystyle \int_G \omega(x)^{-q'} dx < \infty,$
\end{enumerate}
where $q'$ is conjugate of $q$, that is, $\displaystyle \frac{1}{q}+\frac{1}{q'}=1$, and $\ast$ denotes the convolution.
    \end{enumerate}
\end{definition}   
\begin{lemma}[examples \cite{FeiWeight1979, GroWeight2007, Ku}] \  
\begin{enumerate}
    \item  Consider the following class of weight functions:
$$\omega (x)=\omega_{a,b,s,t}(x) = e^{a|x|^b} (1+|x|)^s \left( \log (e+|x|) \right)^t\  \quad  (x \in \mathbb R^d, a,b,s,t \geq 0).$$
    If $0 \leq b \leq 1,$ then  $\omega_{a,b,s,t}$ is a weight (submultiplicative). If $0 \leq b <1,$ then   $\omega_{a,b,s,t}$ satisfies the Beurling-Domar condition.
    \item  The polynomial weight: $\omega_s(x)=(1+|x|)^s$ for sufficiently large $s$ and exponential weight: $\omega(x)=e^{|x|^b}$ for $0<b<1$ are $q-$algebra weights on $\mathbb{R}^d$, and consequently their restrictions on $\mathbb{Q}^d$ and $\mathbb{Z}^d$. In particular, if $1<q<\infty$ and $s>d/q'$, then the polynomial weights $\displaystyle \omega_s(x)=(1+|x|^2)^\frac{s}{2}$ are $q-$algebra weights. See  \cite{Grab,Kerman,Ku}.
    \item Even on $\mathbb{Q}_p$, the $p-$adic number field with $p-$adic norm $|\cdot|_p$, regarding it as locally compact abelian group with addition as binary operation, one may define the above weights by replacing $|\cdot|$ by $|\cdot|_p$.
    \item  The constant weights, in particular $\omega\equiv1$, can not  be a $q-$algebra weight for any $q>1$. 
\end{enumerate}
\end{lemma}

\begin{remark} \label{irw} It is worth noting the motivation  behind  the proposed weights. 
\begin{enumerate}
    \item The  submultiplicative property on $\omega$ ensures that $L^1_\omega(G)$  is an algebra under convolution; while subconvolutive ensures  that $L^q_\omega(G)$ $(1<q< \infty)$ is  an algebra under convolution. See \cite{Ku, GroWeight2007}.
    \item The Beurling-Domar condition was discovered by Domar  in \cite{domar}, which guarantees the existence of test function with  the desired  compact support (see Theorem \ref{Domar}).  This is indispensable when we wish to employ   the  localization technique (e.g. partition of unity) in the proof of our main result. See  Lemmas \ref{lbd} and \ref{bdtobd}. 
\end{enumerate}  
\end{remark}
Let $(\cX,\|\cdot\|)$ be a commutative unital complex Banach algebra. For a $\cX-$valued strongly measurable function $f$ on $G,$  the weighted Lebesgue space norm is given by 
\[\|f\|_{L^q_\omega(G,\cX)}=\|f\|_{L^q_\omega}= \left( \int_G \|f(x)\|^q \omega(x)^q dx \right)^\frac{1}{q},\]
where $dx$ denotes the Haar measure of $G$. The space $L^1_\omega(G,\cX)$ is a complex Banach algebra with the norm $\|\cdot\|_{L^1_\omega(G,\cX)}$ and usual convolution product. The same follows for $L^q_\omega(G,\cX)$ for $q>1$ provided that $\omega$ satisfies $\displaystyle \omega^{-q'}\ast\omega^{-q'}\leq\omega^{-q'}$. 
Note that condition (a) in the definition of $q-$algebra weight is sufficient to make $L^q_\omega(G,\cX)$ an algebra as it gives submultiplicativity of the norm and they are called \textbf{\emph{Beurling algebras}}. But we also impose (b) because it along with H\"older's inequality gives the inclusion $L^q_\omega(G,\cX) \subset L^1(G,\cX)$ which in turn allows us to define Fourier transform of $f\in L^q_\omega(G,\cX)$.  The \textbf{\emph{Fourier transform}} of $f\in L^q_\omega(G,\cX) \subset L^1(G,\cX)$  is given by
$$\widehat{f}(\gamma)=\int_G f(x) \gamma(-x) dx \quad (\gamma\in\Gamma).$$ 
The collection of all Fourier transforms of functions in $L^q_\omega(G,\cX)$ is given by $$A^q_{\omega}(\Gamma,\cX)=\left\{\widehat{f}:\Gamma\to \cX :  f\in L^q_\omega(G,\cX)\right\}.$$ The space $A^q_{\omega}(\Gamma,\cX)$ is a Banach algebra with pointwise operations and norm 
$$\|\widehat{f}\|_{A^q_{\omega}(\Gamma,\cX)}=\|{f}\|_{L^q_\omega(G,\cX)}$$ for $f\in L^q_\omega(G,\cX)$ with $\omega$ being a $q-$algebra weight when $q>1$. 

\begin{remark}[Notations] \label{notation}
\noindent
\begin{enumerate}
    \item[(a)] For convenience, we shall denote a function in $A^q_{\omega}(\Gamma,\cX)$ by $f$, instead of $\widehat{f}$, while its corresponding function in $L^q_\omega(G,\cX)$ by $\widehat f$. It shall be noted that this notation is conventional and arises from the case of discrete $G$ (see \eqref{invFT}).
    
    \item[(b)] If $\omega\equiv 1$, then the corresponding spaces are denoted by $L^q(G,\cX)$ and $A^q(\Gamma,\cX)$.
    
    \item[(c)] If $\cX=\mathbb{C}$, then the corresponding spaces are denoted by $L^q_\omega(G)$ and $A^q_\omega(\Gamma)$.
\end{enumerate}
\end{remark}

To study the composition operator in the vector valued setting, we introduce the following definitions.
\begin{definition}\label{defopfns} 
Let $S$ be a set, $\cX$ and $\mathcal Y$ be Banach algebras, and let $\mathcal{A}$ and $\mathcal{B}$ be two function algebras consisting of functions from $S$ to $\cX$ and $\mathcal Y$, respectively. 
\begin{enumerate}
    \item For $E\subset \cX$ and a function $F:E\to \mathcal Y$, assign a \textbf{\emph{composition operator}} $T_F$ to $F$ given by $$T_F:f \mapsto F(f)=F\circ f,$$ where $F\circ f$ is the composition of the functions $F:E\to \mathcal Y$ and $f:S\to E$.
    \item The operator \textbf{\emph{$T_F$ maps $\mathcal{A}$ to $\mathcal{B}$}} if $T_F(f)\in\mathcal{B}$ for all $f\in\mathcal{A}$ whose range lies in $E$.
    \item  If $\cX=\mathcal Y$ and $\mathcal{A}=\mathcal{B}$, then \textbf{\emph{$T_F$ acts (operates) on $\mathcal{A}$}} if $T_F(f)\in\mathcal{A}$ for all $f\in\mathcal{A}$ whose range lies in $E$.
\end{enumerate}  
\end{definition}
\begin{definition} \label{def:realanalytic} 
Let $d\in\mathbb{N}$, $U\subset\mathbb{C}^d$ be open, and let $\cX$ be a Banach algebra.
\begin{enumerate}
    \item A function $F:U\to \cX$ is \textbf{\emph{real analytic}} in $U$ if for each $z=(x_1+iy_1,x_2+iy_2,\dots,x_d+iy_d)\in U$ if there is $\delta>0$ such that $F$ has series representation of the form  
    \begin{align*}
        F(\xi)=\sum_{\substack{n_j\in\mathbb{N}_0 \\ 1\leq j \leq 2d}} c_{n_1 n_2 \dots n_{2d}} (x_1-k_1)^{n_1} (y_1-l_1)^{n_2} \cdots (x_d-k_d)^{n_{2d-1}} (y_d-l_d)^{n_{2d}}
    \end{align*} 
    which converges absolutely in the norm of $\cX$, where $c_{n_1 n_2\dots n_{2d}}\in \cX$ for $n_1,n_2,\dots,n_{2d}\in\mathbb{N}_0=\mathbb{N}\cup\{0\}$, for all $\xi=(k_1+il_1,k_2+il_2, \dots,k_d+il_d)\in U$ with $|x_n-k_n|<\delta$ and $|y_n-l_n|<\delta$, for $n=1,2,\dots,d$.
    \item For any $E \subset \mathbb{C}^d$, $F$ is real analytic in $E$ if it is real analytic in some open set containing $E$. 
\end{enumerate}  
We refer the reader to \cite{Kran} for real analytic functions on several variables.
\end{definition}

We now briefly mention literature. Let $1\leq q<\infty$, and let $\omega$ be a ($q-$algebra, if $q>1$) weight on $\mathbb{Z}$ satisfying the GRS-condition (See Remark \ref{rem:GRS}). Then Kinani and Bouchikhi \cite{ki} showed that if $F$ is real analytic function, then $T_F$ acts on $A^q_\omega(\mathbb{T},\mathbb{C})$ and its multivariate analogue are obtained by Atef and Kinani in \cite{Kim}. Rudin \cite[Ch. 6 (6.2)]{rb}  established  similar result for locally compact abelian $\Gamma$ for $q=1$ without weight. Feichtinger, Kobayashi, and Sato \cite{HGWL1,HGWL2} has shown similar analogue for $A^q_{\omega_s}(\mathbb{R}^d,\mathbb{C})$ with $\omega_s(x)=(1+|x|^2)^\frac{s}{2}$ for $s\geq0$ with $s>d/q'$ whenever $q>1$.
We summarize them in the following theorem. 

\begin{theorem} \label{ABK}
Let $n,d\in\N$ and $1\leq q <\infty$. If $F:\mathbb C^n\to\mathbb C^n$ be a real analytic function, then 
\begin{enumerate}
    \item $T_F$ acts on $A^1(\Gamma,\mathbb{C})$ with $n=1$ and a locally compact abelian $\Gamma$, 
    \item $T_F$ acts on $A^q_\omega(\mathbb T^d,\mathbb C^n)$, where $\omega$ is a ($q-$algebra, if $q>1$) weight on $\mathbb Z^d$ which satisfies the GRS-condition, and
    \item $T_F$ acts on $A^q_{\omega_s}(\mathbb R^d,\mathbb C)$, where $n=1$ and $\omega_s(x)=(1+|x|^2)^\frac{s}{2}$ with $s\geq0$ if $q=1$ and $s>d/q'$ if $q>1$.
\end{enumerate}
\end{theorem}
The converse of Theorem \ref{ABK} is mainly studied  by Katznelson and Rudin. Helson et al. \cite{helson} showed that if a function $F:[-1,1]\to\mathbb{C}$ is such that $T_F$ acts on $A^1(\Gamma,\mathbb{C})$, then $F$ is real analytic. Moreover, for a function $F:E\to\mathbb{C}$, where $E\subset\mathbb{C}$, $T_F$ acts on $A^1(\Gamma,\mathbb{C})$ if and only if $F$ is real analytic in $E$ and $F(0)=0$ if $\Gamma$ is not compact, see \cite[Ch. 6]{rb}.  We state the classical theorem of Helson, Kahane, Katznelson \cite[p.156]{helson}. See also \cite[Theorem 6.9.2]{rb} for $A^{1}(\mathbb{T},\mathbb{C})$ which was proved in 1959, and later generalized by Rudin \cite {rud} in 1962 for $A^{q}(\Gamma,\mathbb{C})$, where $\Gamma$ is infinite compact abelian group and $1\leq q<2$.  

\begin{theorem}[Helson-Kahane-Katznelson-Rudin] \label{HKKR}
Let $1\leq q<2$, $G$ be a locally compact abelian group, and let $\Gamma$ be the dual group of $G$. Suppose that $T_{F}$ is the composition operator associated to a complex function $F$ on $\mathbb C$. If $T_{F}$ acts on $A^{1}(\Gamma,\mathbb{C})$ or maps $A^{1}(\Gamma,\mathbb{C})$ to $A^{q}(\Gamma,\mathbb{C})$ when $\Gamma$ is compact, then $F$ is real analytic on $\mathbb R^{2}.$ 
\end{theorem}

\subsection{Statement of the main results}
We are now ready to state  our first main result. It can be seen as a \textit{vector valued weighted generalization} of Theorem \ref{HKKR}.
\begin{theorem}\label{mbk}
    Let $1 \leq q < 2$, $G$ be a locally compact (noncompact) abelian group, $\Gamma$ be the dual of $G$, $\omega$ be an admissible  weight on $G$, and let $\cX$ be a unital commutative Banach algebra. Let $T_F$ be the composition operator associated with the function $F:I \to \cX$ for parts \eqref{vecvalq}, \eqref{vecval} or $F:I \to \mathbb{C}^2$ for parts \eqref{twovarq},\eqref{twovar}, where $I$ is a subset of $\mathbb{C}$ or $\mathbb{C}^2$, respectively. Assume that one of the following statements holds true:
    \begin{enumerate}
        \item \label{twovarq} $T_F$ maps $A^1_{\omega}(\Gamma,\mathbb{C}^2)$ to $A^q(\Gamma,\mathbb{C}^2)$,
        \item \label{vecvalq} $T_F$ maps $A^1_{\omega}(\Gamma,\mathbb{C})$ to $A^q(\Gamma,\cX)$,
        \item \label{twovar} $T_F$ acts on $A^q_{\omega}(\Gamma,\mathbb{C}^2)$,
        \item \label{vecval} $T_F$ acts on $A^q_{\omega}(\Gamma,\cX)$,
    \end{enumerate}
    where $G$ is discrete in parts \eqref{twovarq}, \eqref{vecvalq} and $\omega$ is $q-$algebra weight if $q>1$ in parts \eqref{twovar}, \eqref{vecval}.
    Then $F$ is real analytic in $I$ and $F(0)=0$ if $G$ is not discrete.
\end{theorem}

\begin{remark}  Theorem  \ref{mbk}  deserve  several comments.
    \begin{enumerate}
    \item The range of $q$  is sharp in the sense that the same conclusion may not hold true if we take $q=2$. In fact, in this case,  the function  $F(z)=|z|$ (which is not real analytic) serves as the counter example as $A^2=L^2$. It is in this sense we may regard Theorem \ref{mbk} as the strong converse of Wiener-L\'evy theorem.
        \item  Taking $\omega \equiv 1$ and $\cX=\mathbb C$, Theorem \ref{mbk} \eqref{vecvalq}  recovers  Rudin's \cite[Theorem 5]{rud}; see Theorem \ref{HKKR}.
        
        \item  Taking $q=1$, $\omega \equiv 1$ and $\cX=\mathbb C$,  Theorem \ref{mbk} \eqref{vecval}  recovers  Helson et al. \cite[Theorems 1 and 2]{helson}; see Theorem \ref{HKKR}.
        
        \item Theorem \ref{mbk} \eqref{twovarq} and  \eqref{twovar} can be seen as a finite variable weighted analogues of Theorem \ref{HKKR}.
        \item   We have not considered the case where $G$ is compact. In fact, in  this scenario, a weight $\omega$ on $G$ is equivalent to a constant weight as weights are bounded on compact sets \cite[Lemma 1.3.3, pp. 19]{Kan}.  Hence,  Theorem \ref{mbk} follows from  Helson et al. \cite{helson}.
        \item In  parts \eqref{twovarq} and \eqref{vecvalq}, we   could  have two space indices $1$ and $q$ due to  inclusion relation on weighted $\ell^q$ spaces (see \eqref{ineqnormsdis}).   In  parts \eqref{twovar} and \eqref{vecval}, we only have one index $q$ as  if $\Gamma$ is not compact, then \eqref{ineqnormsdis} is not true.
        \item In \eqref{twovarq} and \eqref{twovar}, one can consider $\mathbb{C}^d$ $(d\in\mathbb{N})$ in place of $\mathbb{C}^2$ and get same result with appropriate modifications. Also, in them we have $F:\mathbb{C}^2\to\mathbb{C}^2$ while in Theorems \eqref{vecvalq} and \eqref{vecval}, we have $F:\C\to \cX$, where $\cX$ is a unital Banach algebra because $\mathbb{C}$ can be identified with a subalgebra of $\cX$ both having the same identity $1_\cX$ which is an essential condition in the proof which may not be true for $\mathbb{C}^d$ with $d>2$.
        \item In parts \eqref{twovarq} and \eqref{twovar}, we do not know whether we can replace $A^q(\Gamma,\mathbb{C}^2)$ by $A^q(\Gamma,\cX)$ but if we could, then Theorems \eqref{vecvalq} and \eqref{vecval} will be their special cases. But the key fact in proving these results is that $\mathbb{C}$ can be identified with a subalgebra of $\cX$ via the identification $z\in\mathbb{C}\mapsto z 1_\cX\in \cX$, where $1_\cX$ is the identity of $\cX$. So, if $\cX$ is a unital Banach algebra having $\C^2$ as its subalgebra, both of which have same identity, then we may replace $A^q(\Gamma,\mathbb{C}^2)$ by $A^q(\Gamma,\cX)$ and the more general result follows using the same techniques presented here.
    \end{enumerate}    
\end{remark}
Now we shall turn our attention to the converse of Theorem \ref{mbk}, which is weighted version of Theorem \ref{ABK} for the case of $G=\mathbb{R}^d=\Gamma$. First, we state a definition which puts further conditions on weight that is required for our result.

\begin{definition} \label{def:weight of regular growth}
A weight $\omega$ on $\mathbb{R}^d$ is of \textit{\textbf{regular growth}} if it satisfies the following for all $r\geq1$ and $x\in\mathbb{R}^d$
\begin{enumerate}
    \item[(a)] $\displaystyle \omega\left(\frac{x}{r}\right)\leq \omega(x)$, and
    \item[(b)] $\omega(rx)\leq C r^\alpha \omega(x)$ for some constant $C\geq1$ and $\alpha>0$.
\end{enumerate}
\end{definition}

Observe that the polynomial weights $\omega_s$, in particular the Japanese weight $\omega(x)=(1+|x|^2)^\frac{s}{2}$ for $s>0$, are weights of regular growth on $\mathbb{R}^d$. Note that if $\omega$ is a weight of regular growth, then condition (b) implies that $\omega(x)\leq C|x|^\alpha$ for all $|x|\geq1$, where $C=\sup_{|x|=1} \omega(x)$. Refer to \cite{reiter} for more details on weight of regular growth. We are now ready to state our next  main theorem.
\begin{theorem} \label{suff}
Let $1\leq q<\infty$, $n,d\in\N$, $\omega$ be an admissible ($q-$algebra if $q>1$) weight of regular growth on $\mathbb{R}^d$, and let $\cX$ be a unital commutative Banach algebra.
\begin{enumerate}
    \item \label{ii} If $F:I\to \cX$ $(I\subset\mathbb{C})$ is a real analytic function with $F(0)=0$,  then $T_F$ acts on $A^q_\omega(\mathbb{R}^d,\cX)$.
    \item \label{i} If $F:I\to \mathbb{C}^n$ $(I\subset\mathbb{C}^n)$ is a real analytic function with $F(0)=0$, then $T_F$ acts on $A^q_\omega(\mathbb{R}^d,\mathbb{C}^n)$.
\end{enumerate}
\end{theorem}

A variant  of Theorem \ref{suff}  is  obtained in \cite[Theorem 1.6.17]{reiter} for a weight of regular growth and $q=1$. Feichtinger, Kobayashi and Sato  \cite[Theorem 1.5]{HGWL1}  extended it for  $1\leq q<\infty$   with Japanese weight $\omega_s$ (cf. \eqref{eq:aq2} below and \cite[Theorem 5.1]{HGWL2}). It is worth noting the key difference between these results and Theorem \ref{suff}: In \cite{HGWL1, HGWL2}, the considered functions are in $f\in A^1_\omega(\mathbb{R}^d)$ whose range is contained in a compact subset of $\mathbb{R}^d$ with $\omega$ a weight of regular growth or real valued functions  with Japanese weight $f \in A^q_{\omega_s}(\mathbb{R}^d)$, and $F$ being  analytic  on a neighborhood of $f(\mathbb{R}^d)\cup\{0\}.$ While  Theorem \ref{suff} encompasses all $f \in A^q_{\omega}(\mathbb{R}^d)$ whose range lies in the domain of a real analytic function $F$ and  $\omega$ is a weight of regular growth.

\begin{remark} We do not know whether  Theorem \ref{suff} can be extended to arbitrary $\Gamma$ and not only $\mathbb{R}^d$ as the techniques presented here used the scaling property in $\mathbb{R}^d$. And it is the reason we consider the weights of regular growth which behaves ``nicely" with scaling.
\end{remark}

\subsection{Applications to time-frequency spaces} 
 \subsubsection{Time-frequency spaces}The  second aim  of this paper is to study  composition operators  on  weighted time-frequency  spaces on  $\mathbb R^d$. There  has been a great deal of interest to study dispersive PDE with  Cauchy data in  these low regularity spaces (e.g. modulation, Wiener amalgam,  Fourier amalgam or $\alpha-$modulation spaces). See \cite{Tbenyi, HG4, rsw, moT, nonom, sugimoto2007dilation, tribel}.  Surprisingly,  though composition operators (nonlinear operations)  plays a crucial role in understanding the dynamics of dispersive PDE, our  knowledge on it   is rather limited. See \cite{bhimani2016functions,SugimotoNOPJFA,Sugimoto2011, db, BhimaniITSP}. 

  To state our results, we briefly recall  these spaces.  In 1983, Feichtinger \cite{HG4} introduced a class of Banach spaces, which allow a measurement of space  and Fourier transform variables of a function or distribution $f$ on $\mathbb R^d$ simultaneously using the STFT, the so-called modulation spaces. The \textbf{\emph{short-time Fourier transform (STFT)}} of a function (tempered distribution) $f \in \mathcal{S}'(\mathbb R^d)$ with respect to a non-zero window $\phi \in \mathcal{S}(\mathbb R^d)$ (Schwartz space) is  given by 
\begin{align*} \label{stft}
V_{\phi}f(x,t)= \int_{\mathbb R^{d}} f(w) \overline{\phi(w-x)} e^{-2\pi i t\cdot w} dw, \  (x, t) \in \mathbb R^{2d}
\end{align*} whenever the integral exists. For $1\leq p, q \leq \infty$ and a weight $\omega$ on $\R^d$, the \textbf{\emph{(weighted) modulation spaces}} $M^{p,q}_\omega(\mathbb R^d)$ is defined by the norms:
\[\|f\|_{M^{p,q}_{\omega}}= \left\| \|V_\phi f(x,t)\|_{L_x^p} \omega(t)\right\|_{L_t^q}.\]
 By reversing the order of integration we get another family of spaces, so-called \textit{\textbf{(weighted) Wiener amalgam spaces}} $W_\omega^{p,q} (\mathbb R^{d}),$ which is equipped with the norm
\[\|f\|_{W^{p,q}_{\omega}}= \left\| \|V_\phi f(x,t) \omega (t)\|_{L_t^q} \right\|_{L_x^p}.\]
See also Remark \ref{eqnorm}. The first appearance of amalgam spaces dates back to the work of  Wiener \cite{NW26,NW} in his study of generalized harmonic analysis, where the  amalgam space $W(L^p, \ell^q)=W(L^p, \ell^q) (\R^d)$ is defined by the norm
\[\|f\|_{W(L^p, \ell^q)}= \left( \sum_{n \in \mathbb Z^d} \left( \int_{n+(0,1]^d} |f(x)|^p dx \right)^{q/p} \right)^{1/q}.\]
In the 1980s, Feichtinger \cite{Fei} introduced  a generalization of amalgam spaces.  This enables a vastly wider range of Banach spaces of  functions  or distributions defined on  locally compact  group to be used as a local or global component, resulting in a deep and powerful theory.  Specifically,  he used the notation $W(B,C)$ to define a space of functions  or distributions which are ``locally in  Banach space $B$" and ``globally in Banach space $C$", and called them\textit{ \textbf{Wiener amalgam  type}} spaces. In order to define these spaces precisely we briefly introduce notations.  For any given function $f$ which is locally in $B$  (that is,  $gf\in B$, $\forall g \in C_0^{\infty}(\R^d)),$ we set $f_{B}(x)= \|f g(\cdot -x)\|_{B}$, for some nonzero $g\in C_0^\infty(\rd)$. The space $W(B,C)$ is defined as the space of all functions $f$ locally in $B$ such that  $f_{B}\in C$. The space $W(B, C)$ endowed with the norm  $\|f\|_{W(B, C)}=\|f_{B}\|_{C}.$  Moreover, different choices of nonzero $g\in C^{\infty}_0(\R^d)$  generate the same space and yield equivalent norms, see \cite[Theorem 1]{Fei} and \cite[Proposition 11.3.2]{CH2002}. For an expository introduction to Wiener amalgam spaces with extensive references to the original literature, we refer to \cite{CH2002, CH2007}.

In this paper we consider the  Fourier image of  a particular Wiener amalgam spaces $W(L^p, \ell^q_\omega)$,  which is  known as the  \textit{\textbf{Fourier amalgam spaces}}  $\mathcal{F}W(L^p, \ell^q_\omega)(\R^d)=\widehat{w}_\omega^{p,q} (\mathbb R^{d}).$ Specifically, for  $1\leq p, q \leq \infty$ and a weight $\omega$ on $\Z^d,$ $\mathcal{F}W(L^p, \ell^q_\omega)-$spaces are defined by the norm 
\[\|f\|_{\widehat{w}^{p,q}_\omega}=\|f\|_{\mathcal{F}W(L^p, \ell^q_\omega)}= \bigl\|\|\chi_{n+(0,1]^d} (\xi)\mathcal{F}f(\xi)\|_{L_{\xi}^p(\R^d)} \omega(n)  \bigr\|_{\ell_n^q(\mathbb Z^d)}.\]

Since we consider the spaces above only on $\mathbb{R}^d$, we shall avoid writing it for convenience. That is, we shall simply write $M^{p,q}_\omega$ in place of $M^{p,q}_\omega(\mathbb{R}^d)$ and the same applies to other two spaces.

\subsubsection{Composition operators on time-frequency spaces}
Now we briefly review the literature. Bhimani and Ratnakumar \cite{bhimani2016functions} have completely understood composition operator on $M^{1,1}(\mathbb R^d).$ Later, Kobayashi-Sato \cite[Theorem 1.1]{ks} and Bhimani \cite{db}  generalized this result for  
 $M^{p,1}(\mathbb R^d)$  for $1<p<\infty$. These results are inspired from the classical theorems of   Wiener-L\'evy  and  Helson-Kahane-Katznelson-Rudin. We may summarize them in the following theorem.
 
\begin{theorem}[\cite{bhimani2016functions,ks,db}] \label{bhksdb}
Let $1\leq p< \infty, 1\leq q <2,$  and let 
$$(X,Y)=(M^{p,1},M^{p,q}) \ \text{or} \ (W^{p,1},W^{p,q}).$$
Then $T_F$ maps $X$ to $Y$ if and only if $F$ is real analytic on $\mathbb R^2$ and $F(0)=0$.
\end{theorem}

On the other hand, Sugimoto, Tomita and Wang \cite{Sugimoto2011} studied composition operator on weighted modulation spaces  $M^{p,2}_\omega=M^{p,2}_s$ with  the Japanese (polynomial) weight $\omega_s(\xi)= (1+|\xi|^2)^{s/2}$ and obtained some non-analytic results (cf. \cite{BhimaniITSP}). Later, Kato, Sugimoto and Tomita   \cite{SugimotoNOPJFA} also established non-analytic result. Specifically, they showed that if $F\in C^{\infty}(\mathbb R)$ and $F(0)=0$, then  $T_F$ acts on $M^{p,q}_s(\mathbb R) \ (1\leq p < \infty, 4/3\leq q < \infty, s>1/q'),$ (cf. \cite[Chapter 5]{kassob}). While  Reich, Reissig and Sickel  \cite{Reich2016,Sickel2016} also  studied  composition operator in  modulation spaces with quasi-analytic regularity. Biswas  and Manna \cite{MannaAffine}  established the analogue of  Theorem \ref{bhksdb} in the affine modulation spaces.  Kobayashi and Sato  \cite{Kobayashi2022, KobayashiPSDOP}  established some discrete  analogue of this in $A^1_s(\mathbb T)$ and  so in some modulation spaces.

Taking these considerations into account and keeping the fact that  increasing  interest  of weights in the analysis of PDE, we are inspired to study the weighted analogue of Theorem \ref{bhksdb}.  In fact, we could also include Fourier amalgam spaces, which was recently used in understanding the dynamics of  nonlinear Schr\"odinger equations, see \cite{forlano2020deterministic,BhiStrongIll}.  Specifically, our  second main result is the  following. 
\begin{theorem}\label{tfam}
Let $1\leq p \leq  \infty$, $1\leq q <2$, $\omega$ be admissible weight on $\mathbb{R}^d$ (on $\zd$ for Fourier amalagam space), and let $$(X,Y)=\left(M_{\omega}^{p,1},M^{p,q}\right) \quad \text{or} \quad \left(W_{\omega}^{p,1},W^{p,q}\right) \quad \text{or} \quad  \left(\mathcal{F}W(L^p, \ell^1_\omega), \mathcal{F}W(L^p, \ell^q)\right).$$ Let  $T_{F}$ be the composition operator associated to a complex function $F:I \to \mathbb{C}$, where $I \subset \mathbb C$. Then
\begin{enumerate}
\item \label{nc1}  If  $T_{F}$ maps $X$ to $Y$, then $F$ must be real analytic in $I$. Moreover, $F(0)=0$ if $p<\infty.$
\item  \label{sc}  If $p,q \in[1,\infty)$, $\omega$ is an admissible ($q-$algebra, if $q>1$) weight of regular growth and $F$ is real analytic in $I$ which takes origin to itself, then $T_{F}$  acts on $M_{\omega}^{p,q}$, $W_{\omega}^{p,q}$ and $\mathcal{F}W(L^p, \ell^q_\omega)$.
\end{enumerate}
\end{theorem}

\begin{corollary}\label{corapp} Let $1\leq p \leq\infty$, $1\leq q<2$, $\omega$ be an admissible weight of regular growth which is $q-$algebra weight for $q>1$, and let $X=M_{\omega}^{p,q}$, $W_{\omega}^{p,q}$ or $\mathcal{F}W(L^p, \ell^q_\omega).$ Then
\begin{enumerate}
    \item $T_{F}$ acts on $X$ if and only if  $F$ is real analytic on $\mathbb R^2$ and $F(0)=0.$
    \item   there exists $f\in M^{p,1}_{\omega} $ such that (power-type nonlinearity) $|f|^{\alpha}f \notin M^{p,q}$ for $\alpha \in (0, \infty)\setminus 2\mathbb N$.
\end{enumerate}
\end{corollary}
The power type $|u|^{\alpha}u$ and exponential type $e^{u|u|}-1$ nonlinearities are ubiquitous in  PDE. See for e.g. the  excellent survey article by Ruzansky-Sugimoto-Wang \cite{ruzhansky2012modulation}. In fact, it is known that the algebra property of  is importance in PDE, and it was in  this context they had raised open question \cite[Question 7.1]{ruzhansky2012modulation}: Does 
\begin{eqnarray}\label{wop}
 \||u|^{\alpha}u\|_{M^{p,1}}\leq \|u\|_{M^{p,1}}^{\alpha+1} \quad (\alpha \in (0, \infty)\setminus 2 \mathbb N)   
\end{eqnarray}
holds for all $u \in M^{p,1}$? The authors in  \cite{bhimani2016functions, ks, db}   resolved this question in the unweighted case.  We would like to point out that   Corollary  \ref{corapp}  reveals that  estimate \eqref{wop}  cannot be hold be true  even in the weighted modulation spaces. As a consequence, this shed light  to the standard approach  employed to study well-posedness theory for dispersive PDE in the realm of \textit{weighted} modulation spaces.  

\begin{remark}  Theorem \ref{tfam}  deserves several comments.
\begin{enumerate}
\item A variant of  Theorem \ref{tfam} \eqref{sc} with Japanese weight has been obtained by Feichtinger, Kobayashi and Sato in \cite[Theorem 5.1]{HGWL2}.
\item The modulation  and Wiener amalgam spaces can be defined on general locally compact abelian group, refer to \cite{feichtinger1983modulation}. The analogue of Theorem \ref{tfam} should hold true for more general locally compact abelian  group  with appropriate modifications.  However, here for the sake of clarity and keeping the applications in PDE in mind, we restrict ourself to the Euclidean space.
\item Theorem  \ref{tfam} can also be shown for $\C^d$ valued and Banach algebra valued distributions using proper modifications and using Theorem \ref{mbk} \eqref{twovarq} and Theorem \ref{mbk} \eqref{vecvalq} as per the case.

  
\end{enumerate}
\end{remark}

\subsection{Proof techniques and novelties}
The proof technique for Theorem \ref{mbk} is inspired mainly by the work of Katznelson \cite{helson} and Rudin \cite{rud}.  Here, first we show the result for discrete $G$ and then using group structure theory extend it to locally compact $G$.  To  establish  the proof of Theorem \ref{mbk},  the following key steps are in order:
\begin{itemize}
    \item[--] $T_F$ is locally bounded at each point of $\Gamma$. See Definition \ref{def:lbd} and Lemma \ref{lbd}. In fact, at this stage  we  shall invoke   Beurling-Domar condition effectively. See Remark  \ref{irw}. 
    \item[--] $T_F$ maps a certain neighborhood of 0 in $A^1_\omega(\Gamma,\cX)$ to a bounded set in $A^q(\Gamma,\cX)$. See Lemma \ref{bdtobd}. This step involves the idea of   partition of unity together with the aid of generalized  Wiener's  lemma due to  Bochner and Phillips (Theorem  \ref{kb}). As a consequence, we shall get  the  continuity of $F$.
    \item[--]  To establish the  real analyticity  of $F$, we shall need technical Lemma \ref{sup}. The restriction on  $q$ comes due to this step, see Remark \ref{whyq<2}.  The idea is to get the sufficient decay of the Fourier coefficient of $F_1(s)= F(r\sin s)$, see \eqref{def:F_1}, \eqref{fsf} and \eqref{fsfdecay}.  This will allow us to  extend the $F_1$ holomorphically  on the horizontal strip (containing  the  real line), which will eventually gives the real analycity of $F$.
\end{itemize}
Now we shall briefly  comment on Lemma \ref{sup}, which   gives the crucial  lower bound of the norm $\|e^{inf}\|_{A^q}$ for trigonometric polynomial $f$ form a sphere in $A^q_{\omega}$. To this end, we have to improve the existing similar lower bound of Katznelson and  Rudin as we are working with the smaller  weighted space $A^q_{\omega}$. The difficulty rises due to the weight, for instance,  the norm in $L^q_\omega(G,\cX)$ is not translation invariant. We  resolve this by choosing appropriate function with the aid of adjusted weight. We note that the  technique of Katznelson fails in our setup, however, we  could implement  some of the key ideas developed by Rudin \cite{rud}. As mentioned earlier, the one of the key fact was $\C$ can be seen as subalgebra of $\cX$ and it was used for this Lemma \ref{sup}. The other key idea in our approach is to use the facts that $A^q_\omega(\Gamma,\mathbb{C})\subset A^q_\omega(\Gamma,\cX)$ and all the weighted spaces are contained in $A^1(\Gamma,\cX)$.

The technique used to prove Theorem \ref{tfam} \eqref{nc1} traces back to the work in \cite{bhimani2016functions, ks, db}. The key idea is to translate the problem from the Euclidean space to the torus (see Proposition \ref{promaps}). To this end, we identify time-frequency spaces on torus with Fourier algebra; see Appendix \ref{apd} and Proposition \ref{B1}.

For the sufficiency parts Theorems \ref{suff} and \ref{tfam} \eqref{sc}, the main aspect is to use the algebra structure of the spaces in consideration along with the classical result of Domar (Theorem \ref{Domar} and Remark \ref{Domarvector}). We shall combine the ideas presented in \cite{rud, reiter} to achieve our results. As stated earlier, these results are obtained only for $\mathbb{R}^d$ as we shall make use of the scaling argument that cannot be employed for arbitrary $G$. Since scaling argument is used, we also restrict ourselves to a smaller class of weights, namely weights of regular growth. 
\subsection{Further remarks}\label{rem:openque}
The proposed weights in our main results are worthy of multiple comments and naturally prompt interesting open questions.

\begin{remark} \label{rem:GRS}
    In recent literature, a weight $\omega$ is admissible if it satisfies the \emph{Gelfand-Raikov-Shilov (GRS) condition} $$\lim_{n\to\infty} \omega(nx)^\frac{1}{n}=1 \quad (x\in G).$$ Note that all weights which satisfy BD-condition also satisfies GRS-condition. But the converse is not true. For example, if $\displaystyle \omega(x)=e^\frac{|x|}{\log(e+|x|)}$ for $x\in\mathbb{R}$, then it satisfies GRS-condition but not BD-condition. 
\end{remark}

\begin{remark}
     The weighted analogues of Wiener's theorem for $G=\mathbb{Z}^d$ has been established in \cite{Bhatt,BhDeDa,PAD} and the  references therein. While the vector valued analogues in \cite{Bo,kb,kb2}  for admissible weights with GRS-condition and not BD-condition. But, here we consider admissible weights to satisfy BD-condition because Domar \cite{domar} showed that the algebra $A^1_\omega(\Gamma)$ contains functions with compact support if and only if $\omega$ satisfies BD-condition, and this fact is required in our proofs. We do not know whether we can replace BD-condition with GRS-condition here.
\end{remark}

\begin{remark}
  Wiener and L\'evy theorems for $G=\mathbb{Z}^d$ has been generalized for larger class of weights, which satisfy the GRS-condition, see \cite{PAD,Kim,Bhatt,BhDeDa,ki,kb,kb2} and references there in. It shall be noted that then even Wiener's theorem may not be true for $A^1_\omega(\mathbb{T})$ if $\omega$ does not satisfy GRS-condition. For example, consider the weight $\omega(n)=e^{|n|}$ $(n\in\mathbb{Z})$ which is not GRS as $\lim_{n\to\infty} \omega(n)^\frac{1}{n}=e$. Then the function $f(z)=2-z$ is in $A^1_{\omega}(\mathbb{T})$ and $f(z)\neq 0$ for all $z\in\mathbb{T}$ but $\frac{1}{f}$ is not in $A^1_\omega(\mathbb{T})$, or equivalently, $T_F$ does not acts on $A^1_{\omega}(\mathbb{T})$ for the analytic function $F(z)=\frac{1}{z}$.   
\end{remark}

\begin{remark} The spaces $A^q(\Gamma,\cX)$, $M^{p,q}_\omega(\R^d)$, $W^{p,q}_\omega(\R^d)$ and $\widehat{w}^{p,q}_\omega(\R^d)$ are not algebra in general if $q>1$ but becomes an algebra when we consider $\omega$ to be a $q-$algebra weight (see Proposition \ref{proalgebra}). So, an analogue of L\'evy's theorem follows for them for analytic function, that is, if $F:\mathbb C \to \mathbb C$ is real entire function\footnote{Notice  $F$ is real analytic everywhere on $\mathbb R^2$, does not imply that $F$ is real
entire. See \cite[Remark 3.9]{bhimani2016functions}.}, then $T_F$ acts on $M_\omega^{p,q}$, $W_\omega^{p,q}$ and $\widehat{w}_\omega^{p,q}$ whenever $\omega$ is an admissible $q-$algebra weight. See \cite[Theorem 3.9]{bhimani2016functions}. 
\end{remark}

\begin{remark}\label{eqnorm}  
    The definitions of modulation and amalgam spaces are independent of the choice of the window $\phi$, in the sense that different window functions yield equivalent modulation space norms (cf. \cite[Proposition 11.3.2 (c), p.233]{GroWeight2007}) and equivalent Wiener amalgam space norms.  Note that 
$W^{p,q}_{\omega}=W(\mathcal{F}L^q_\omega,L^p)$ and $M^{p,q}_{\omega}=W(\mathcal{F}L^p,L^q_\omega)$.   There is also an equivalent definition of modulation spaces  via frequency-uniform decomposition operators; which is quite similar in the spirit of Besov spaces. See \cite[Proposition 2.1]{wang2007global}, \cite{feichtinger1983modulation}.
\end{remark}

The organization of the article is as follows: In Section \ref{sec:pre}, basic definitions and known results required are provided. Section \ref{sec:prilem} is dedicated to lemmas required to prove Theorem \ref{mbk}. In Section \ref{sec:nec} and Section \ref{sec:suff}, we derive the necessary (Theorem \ref{mbk} and Theorem \ref{tfam} \eqref{nc1}) and sufficient conditions (Theorem \ref{suff} and Theorem \ref{tfam} \eqref{sc}), respectively.  

\section{Preliminaries} \label{sec:pre}
\noindent
\textbf{Notations}. The notation $A \lesssim B $ means $A \leq cB$ for a some constant $c > 0 $, whereas $ A \asymp B $ means $c^{-1}A\leq B\leq cA $ for some $c\geq 1$. The symbol $A_{1}\hookrightarrow A_{2}$ denotes the continuous embedding  of the topological linear space $A_{1}$ into $A_{2}$. For $1<q<\infty$, $q'$ will be its conjugate index, that is, $\frac{1}{q}+\frac{1}{q'}=1$. 

First, we give brief details about the Banach algebras. A normed linear space $(\cX,\|\cdot\|)$ over $\mathbb{C}$ is a \textit{\textbf{(complex) normed algebra}} if $\cX$ is an algebra and the norm on $\cX$ is submultiplicative, that is, $\|xy\|\leq \|x\|\|y\|$ for all $x,y\in\cX$. We say a normed algebra $\cX$ is
\begin{itemize}
    \item \textbf{\textit{unital}} if there is an element $1_\cX\in\cX$ such that $1_\cX x=x=x1_\cX$ for all $x\in\cX$. The element $1_\cX$ is called the \textbf{\textit{unit or identity}} of $\cX$.
    \item \textbf{\textit{commutative}} if $xy=yx$ for all $x,y\in\cX$.
    \item \textbf{\textit{Banach}} if it is complete in the topology induced by the norm $\|\cdot\|$.
\end{itemize}
All the algebras considered here are complex algebras except when stated otherwise. Throughout the paper, if not mentioned, $\cX$ is a unital commutative complex Banach algebra with unit element denoted by $1_\cX$, for simplicity, and the norm on $\cX$ will be denoted by $\|\cdot\|$. 

A subset $\cX_1\subset\cX$ is a \textbf{\textit{subalgebra}} of $\cX$ if $\cX_1$ is itself an algebra with same operations as of $\cX$. 
\begin{remark}\label{CsubalgebraX}
Note that if $\cX_1$ is unital subalgebra of $\cX$, then it is not necessary that both have same unit. Take for example $\mathbb{C}\times\{0\}$ as a subalgebra of $\mathbb{C}^2$ . Then unit element of $\mathbb{C}^2$ is $1_\cX=(1,1)$ but that of $\mathbb{C}\times\{0\}$ is $1_{\cX_1}=(1,0)$. But we can always identify $\mathbb{C}$ as unital subalgebra of given unital complex algebra $\cX$ both having the same unit by the map $\mathbb{C}\ni z\mapsto 1_\cX z\in\cX$.
\end{remark}

\begin{lemma} \label{lem:weighted Fourier algebras inclusions}
Let $1\leq q<\infty$, $\omega_1$ and $\omega_2$ be ($q-$algebra, if $q>1$) weights on $G$ such that $\omega_2\leq K \omega_1$ for some $K>0$, and let $\cX_1$ be a subalgebra of $\cX$.
    \begin{enumerate}
        \item $A^q_{\omega_1}(\Gamma,\cX_1) \hookrightarrow A^q_{\omega_2}(\Gamma,\cX) \hookrightarrow A^1(\Gamma,\cX)$  with the norm inequality
        \begin{align} \label{ineqnorms}
        \|f\|_{A^1(\Gamma,\cX)} \lesssim  \|f\|_{A^q_{\omega_2}(\Gamma,\cX_1)} \leq K \|f\|_{A^q_{\omega_1}(\Gamma,\cX)}
        \end{align}
        \item  If $1\leq p < q <\infty$, $G$ is a discrete abelian group, then we have $A^p_{\omega_1}(\Gamma,\cX) \hookrightarrow A^q_{\omega_1}(\Gamma,\cX) \hookrightarrow A^q_{\omega_2}(\Gamma,\cX)$ with
        \begin{align} \label{ineqnormsdis} 
        \|f\|_{A^q_{\omega_2}} \leq  K \|f\|_{A^q_{\omega_1}} \leq  K \|f\|_{A^p_{\omega_1}}.
        \end{align}
    \end{enumerate}
\end{lemma}

If $\Gamma$ is compact, the Fourier transform is invertible in the sense that given $f\in A^q_\omega(\Gamma,\cX)$, its corresponding function $\widehat{f}\in\ell^q_\omega(G,\cX)$ is given by 
\begin{equation} \label{invFT}
    \widehat{f}(x)=\int_\Gamma f(\gamma) \gamma(x) d\gamma \quad (x\in G),
\end{equation}
where $d\gamma$ is the normalized Haar measure on $\Gamma$ (cf. Remark \ref{notation}). This also allows us to take $\omega\equiv1$. Now, if $f\in A^1_\omega(\Gamma,\cX_1)$ and $g\in A^q(\Gamma,\cX)$, then by H\"olders inequality, we have
\begin{equation} \label{Holder relation}
    \|fg\|_{A^q}\leq \|f\|_{A^1} \|g\|_{A^q} \leq \|f\|_{A^1_\omega} \|g\|_{A^q}.
\end{equation}

Lastly, we state some theorems that will be used here.

\begin{theorem}[Urysohn's lemma] \label{ury} \cite[p. 131]{fol}
Let $X$ be a normed space, $V$ be an open subset of $X$, and let $K\subset V$ be a compact set. Then there exists a continuous function $f:X\to[0,1]$ such that $\supp(f)\subset V$ and $f(x)=1$ for all $x\in K$.
\end{theorem}

\begin{theorem}[$C^\infty-$Urysohn's lemma] \label{Cury} \cite[p. 245]{fol}
Let $K\subset \mathbb{R}^d$ be a compact set, and let $V$ be an open set containing $K$. Then there exists a function $f\in C_c^\infty(\mathbb{R}^d)$, set of infinitely differentiable functions having compact support, such that $0\leq f\leq1$, $\supp(f)\subset V$ and $f(x)=1$ for all $x\in K$.
\end{theorem}

Next, we state a result by Domar which can be seen a version of Urysohn's lemma for the weighted Fourier algebra. 

\begin{theorem} [Domar] \label{Domar} \cite[Lemma 1.24 with Theorem 2.11]{domar}
Let $G$ be a locally compact abelian group, $\Gamma$ be its dual, and let $\omega$ be a weight on $G$ satisfying Beurling-Domar condition \eqref{BDcond}. If $U$ and $K$ are open and compact subset of $\Gamma$, respectively, with $K\subset U$, then there is $f\in A^1_\omega(\Gamma,\mathbb{C})$ such that $0\leq f \leq 1$, $\supp(f)\subset V$ and $f(\gamma)=1$ for all $\gamma\in K$.
\end{theorem}

We state a remark that will be used several times in this paper. 
\begin{remark} \label{Domarvector} 
We can easily extend the above result of Domar for vector valued Fourier algebra. Let $\cX$ be a unital Banach algebra. Then with all details as above in Theorem \ref{Domar} just consider the function $\widetilde{f}=f1_\cX$. This gives $\widetilde f\in A^q_\omega(\Gamma,\cX)$ such that $0\leq \|\widetilde f(\gamma)\| \leq 1$ $(\gamma\in\Gamma)$, $\supp(\widetilde f)\subset V$ and $\widetilde f(\gamma)=1_\cX$ for all $\gamma\in K$. 
\end{remark}

The vector valued $p-$power analogues of Wiener's theorem are obtained in \cite{kb}. Using it techniques along with the classical result of Bochner and Phillips \cite[Theorem 3 and 4]{Bo}, we get the following result.

\begin{theorem}[Wiener-Bochner-Phillips] \label{kb}
Let $1\leq q<\infty$, $G$ be a discrete abelian group, $\Gamma$ be the (compact) dual group of $G$, and let $\omega$ be an admissible $q-$algebra weight on $G$. Let $\cX$ be a unital Banach algebra, and let $f \in A^q_\omega(\Gamma,\cX)$. If $f(z)$ is invertible for all $z \in \Gamma$, $\frac{1}{f}\in A^q_\omega(\Gamma,\cX)$.
\end{theorem}

The following theorem is due to Rudin and is useful in proving the continuity of $F$ in Theorem \ref{mbk}.

\begin{theorem}\cite[Theorem 4]{rud} \label{thm:rud 4}
    Let $1\leq q <2$, $G$ be a discrete abelian group with compact dual $\Gamma$, and let $\chi_{_E}$ denote the characteristic function of $E\subset \Gamma$. If $S$ is the set of all closed subsets of $\Gamma$, then $\sup_{E\in S} \|\chi_{_E}\|_{A^q}=\infty$.
\end{theorem}

The following result is due to Feichtinger and we state it in the form required.

\begin{theorem} \cite[Corollary 2]{Fei} \label{coralgebra}
Let $1\leq q<\infty$, $G$ be a locally compact abelian group, and let $\omega$ be a weight on $G$ which satisfies $\omega^{-q'}\ast\omega^{-q'}\leq \omega^{-q'}$ for $q>1$. Then $W(L^p(G),L^q_\omega(G))$ is a Banach ideal in $W(L^1(G),L^q_\omega(G))$. In particular, $W(L^p(G),L^q_\omega(G))$ is a Banach algebra with convolution for all $1\leq p \leq \infty$.
\end{theorem}


\begin{proposition}[Algebra property]\label{proalgebra}
Let $p,p_{i} \in [1, \infty]$, $q,q_{i} \in [1, \infty]$, $(i=0,1,2)$  satisfy $\frac{1}{p_1}+ \frac{1}{p_2}= \frac{1}{p_0}$ and $\frac{1}{q_1}+\frac{1}{q_2}=1+\frac{1}{q_0}$, and let $\omega$ be a weight on $\mathbb{R}^d$. Then
\begin{enumerate}
    \item \label{ams} $M^{p_1, q_1}_\omega \cdot M^{p_{2}, q_{2}}_\omega \hookrightarrow M^{p_0, q_0}_\omega$ with norm inequality $$\|f g\|_{M^{p_0, q_0}_\omega}\lesssim \|f\|_{M^{p_1, q_1}_\omega} \|g\|_{M^{p_2,q_2}_\omega}.$$

    \item \label{aws} $W^{p_1, q_1}_\omega \cdot W^{p_{2}, q_{2}}_\omega \hookrightarrow W^{p_0, q_0}_\omega$ with  norm inequality $$\|f g\|_{W^{p_0, q_0}_\omega}\lesssim \|f\|_{W^{p_1, q_1}_\omega} \|g\|_{W^{p_2,q_2}_\omega}.$$
    
    \item \label{algebra} If $1\leq q<\infty$ and $\omega$ satisfies $\omega^{-q'}\ast\omega^{-q'}\leq \omega^{-q'}$ for $q>1$, then $M_\omega^{p,q}$, $W_\omega^{p,q}$ and $\widehat{w}_\omega^{p,q}$ are Banach algebras with pointwise operations.
    
    \item \label{mwaminc} $W^{p,q_{1}}_\omega \hookrightarrow L^{p}_\omega \hookrightarrow W^{p,q_{2}}_\omega$ and $M^{p,q_{1}}_\omega \hookrightarrow L^{p}_\omega \hookrightarrow M^{p,q_{2}}_\omega$ holds for $q_{1}\leq \text{min} \{p, p'\}$ and $q_{2}\geq \text{max} \{p, p'\}$ with $\frac{1}{p}+\frac{1}{p'}=1.$
    
    \item \label{faminc} $\widehat{w}_\omega^{p,q}$ is an $A^1_\omega-$module, that is, $\|fg\|_{\widehat{w}_\omega^{p,q}} \leq \|f\|_{A^1_\omega} \|g\|_{\widehat{w}_\omega^{p,q}}$.
\end{enumerate}
\end{proposition}
\begin{proof}
\eqref{ams}, \eqref{aws} and \eqref{mwaminc}: These results for $\omega\equiv1$ are obtained in \cite{baoxiang2006isometric}, \cite[Corollary 2.7]{Tbenyi2009}, \cite[Lemma 2.2]{E1} and \cite{moT}.  
Following same techniques and using the fact $\|f\|_{L^q_\omega}=\|f\omega\|_{L^q}$, we get the desired result.

\eqref{algebra}: Applying Fourier transform on the following inclusion which follows from \cite[Theorem 3]{Fei} $$W(\mathcal{F}L^p,L^q_\omega) \ast W(\mathcal{F}L^p,L^q_\omega) \subset W(L^1,L^q_\omega) \ast W(\mathcal{F}L^p,L^p_\omega) \subset W(\mathcal{F}L^p,L^q_\omega)$$ and using Theorem \ref{coralgebra} it follows that $M^{p,q}_\omega$ is a Banach algebra. Also, $M^{p,q}_\omega=\mathcal{F}W^{p,q}_\omega$ gives desired result for $W^{p,q}_\omega$. Since $\widehat{w}^{p,q}_\omega=\mathcal{F}W(L^p,\ell^q_\omega)$, by Theorem \ref{coralgebra} it follows that $\widehat{w}^{p,q}_\omega$ Banach algebra.

\eqref{faminc}: Follows from Young's inequality.

\end{proof}

\section{The principal lemmas} \label{sec:prilem}
In this section, we shall provide some lemmas that will be used to prove our results.
Throughout this section, we shall assume that  $1 \leq q < 2$, $G$ is a discrete abelian group, $\Gamma$ is the dual of $G$, and $\omega$ is an admissible weight on $G$. 

\begin{definition} \label{def:lbd}
Let $S$ be a set, $\cX$ and $\mathcal Y$ be Banach algebras, and let $\mathcal{A}$ and $\mathcal{B}$ be two function algebras consisting of functions from $S$ to $\cX$ and $\mathcal Y$ respectively. Let  $F:E\subset \cX \to \mathcal Y$ be a function, and  $$T_F(f)=F(f)=F\circ f \quad \text{for} \ f\in \mathcal{A}.$$ The map $T_F$ is \textbf{locally bounded} at $s\in S$ if there are $L>0$, $\eta>0$ and a neighborhood $V_s$ of $s$ such that $$\|T_F(f)\|_{\mathcal{B}}<L$$ for all $f\in \mathcal{A}$ whose range is contained in $E$, support lies in $V_s$ and $\|f\|_{\mathcal{A}} < \eta$. 
\end{definition}

\begin{remark}\label{ter}  For $\gamma' \in\Gamma$, let $L_{\gamma'}$ be the left translation of $\Gamma$ given by $(L_{\gamma'}f)(\gamma) = f(\gamma'^{-1} \gamma)$ for $\gamma \in \Gamma,$  $f\in A^1_{\omega}(\Gamma,\cX_1)$.
 Note that $\|\cdot\|_{A^1_\omega}$ is a left translation invariant. In fact, for  $f\in A^1_{\omega}(\Gamma,\cX_1)$, we have \begin{align*}
    \|L_{\gamma'}(f)\|_{A^1_\omega}=\|\widehat{L_{\gamma'}(f)}\|_{\ell^1_\omega}&=\sum_{x\in G} \|\widehat{L_{\gamma'}(f)}(x)\|\omega(x) \\
    &=\sum_{x\in G} \left\|\int_\Gamma f(\gamma'^{-1} \gamma) \gamma(x) d\gamma \right\| \omega(x) \\
    &=\sum_{x\in G} \left\|\int_\Gamma f(\gamma'^{-1} \gamma) (\gamma'^{-1}\gamma)(x)  \gamma'(x) d\gamma \right\| \omega(x) \\
    &=\sum_{x\in G} |\gamma'(x)| \left\|\int_\Gamma f(\gamma'^{-1} \gamma) (\gamma'^{-1}\gamma)(x) d\gamma \right\| \omega(x) \\
    &=\sum_{x\in G} \|\widehat{f}(x)\|\omega(x) \\
    &= \|f\|_{A^1_\omega}.
\end{align*}
 Similar results follow for all right translations $R_{\gamma'}$, for $\gamma'\in\Gamma$, given by $R_{\gamma'}(f)(\gamma)=f(\gamma \gamma')$ for $\gamma \in \Gamma$ and $f\in A^1_{\omega}(\Gamma,\cX_1)$.
\end{remark}

Now, we state a lemma that states that $T_F$ is locally bounded. It is achieved through the contradiction argument using the Domar's version of Urysohn's lemma for weighted Fourier algebra Theorem \ref{Domar} and the fact that the norms $\|\cdot\|_{A^1_\omega}$ and $\|\cdot\|_{A^q}$ are translation invariant for $\Gamma$. It should be duly noted that the (weighted) norm $\|\cdot\|_{L^1_\omega}$ is not translation invariant for translations of $G$.

\begin{lemma} \label{lbd}
Let $\cX_1$ be a unital subalgebra of a unital Banach algebra $\cX$ both having the same unit $1_\cX$, $E \subset \cX_1$ be such that it contains the unit ball of $\cX_1$, and $F : E \to \cX$ be a function such that $T_F$ maps $A^1_{\omega}(\Gamma,\cX_1)$ to $A^q(\Gamma,\cX)$. Then $T_F$ is locally bounded at each $\gamma \in \Gamma$.
\end{lemma}
\begin{proof}
    We may assume that $F(0)=0$ as if not, then replace $F$ by $F-F(0)$. Note that $L_{\gamma'}$ and $T_F$ commutes on $ A^1_{\omega}(\Gamma,\cX_1)$. In fact, we have  $$(T_F \circ L_{\gamma'})(f)(\gamma)=F\circ f(\gamma'^{-1}\gamma)=(L_{\gamma'} \circ T_F )(f)(\gamma), \quad  \forall \  \gamma \in \Gamma, f\in A^1_{\omega}(\Gamma,\cX_1). $$ Thus, $T_F$ commutes with all translations of $\Gamma$ and $\|\cdot\|_{A^1_\omega}$ and  $\|\cdot\|_{A^q}$ are   translation invariant (see Remark \ref{ter}).

Now, if $T_F$ is locally bounded at some $\gamma\in\Gamma$, then there are $L>0$, $\eta>0$ and a neighborhood $V$ of $\gamma$ such that $\|T_F(f)\|_{A^q}<L$ for all $f\in A^1_{\omega}(\Gamma,\cX_1)$ whose range is contained in $E$, support lies in $V$ and $\|f\|_{A^1_{\omega}} < \eta$. Let $\gamma'\in\Gamma$ with $\gamma\neq\gamma'$, and let $V'$ be a translate of $V$ such that $\gamma'\in V'$. Then the facts that $T_F$ commutes with all translations of $\Gamma$ and the norms $\|\cdot\|_{A^1_{\omega}}$ and $\|\cdot\|_{A^q}$ are translation invariant implies that $\|T_F(g)\|_{A^q}<L$ for all $g\in A^1_{\omega}(\Gamma,\cX_1)$ whose range is contained in $E$, support lies in $V'$ and $\|g\|_{A^1_{\omega}} < \eta$, that is, $T_F$ is locally bounded at $\gamma' \in \Gamma$. Since $\gamma'$ was chosen arbitrarily, it means that $T_F$ is locally bounded at each $\gamma\in\Gamma$. Thus, $T_F$ can be either locally bounded at every point of $\Gamma$ or at no point of $\Gamma$.

For each $n\in\mathbb{N}$, choose disjoint open sets $V_n$ in $\Gamma$ such that there are nonempty sets $W_n$ with $\overline{W_n}\subset V_n$. By Theorem \ref{Domar} and Remark \ref{Domarvector}, there are $\phi_n\in A^1_{\omega}(\Gamma,\cX)$ with $0\leq\|\phi_n(\gamma)\|\leq1$ $(\gamma\in\Gamma)$, $\phi_n(\gamma)=1_\cX$ for all $\gamma\in W_n$ and $\phi_n$ vanishes outside $V_n$.  
 Suppose that $T_F$ is not locally bounded at any $\gamma \in \Gamma$.
  Then, for each $n\in\mathbb{N}$, there are $f_n \in A^1_{\omega}(\Gamma,E)$ such that 
\begin{enumerate}
	\item $f_n$ vanishes outside $W_n$,
	\item $\|f_n\|_{A^1_{\omega}}<\frac{1}{n^2}$, and
	\item $\|F(f_n)\|_{A^q}>n \|\phi_n\|_{A^1_{\omega}}$.
\end{enumerate}
Let $f=\sum_{n\in\mathbb{N}} f_n$. Then $f\in A^1_{\omega}(\Gamma,E)$ as for each $\gamma\in \Gamma$, there is $n\in\mathbb{N}$ such that $\gamma\in W_n$ and $f(\gamma)=f_n(\gamma)\in E$. Since $\phi_n(\gamma)=1_\cX$ $(\gamma\in W_n)$, $\phi_n$ vanishes outside $V_n$, $f_n$ vanishes outside $W_n$ and $F(0)=0$, it follows that $$\phi_n(\gamma) F(f(\gamma))= \phi_n(\gamma) F(f_n(\gamma))= F(f_n(\gamma)) \quad (n\in\mathbb{N}, \gamma \in \Gamma).$$ This along with \eqref{Holder relation} gives \begin{align*} \|\phi_n\|_{A^1_\omega} \|F(f)\|_{A^q} \geq \|\phi_n F(f)\|_{A^q} = \|F(f_n)\|_{A^q} > n \|\phi_n\|_{A^1_{\omega}}. \end{align*} Thus $\|F(f)\|_{A^q}>n$ for all $n\in\mathbb N$ which is not possible as $T_F$ maps from $A^1_{\omega}(\Gamma,\cX_1)$ to $A^q(\Gamma,\cX)$. Hence, $T_F$ is locally bounded at each $\gamma\in\Gamma$.
\end{proof}
Our next lemma tells us that $T_{F}$ maps certain neighborhood  $V$  to a bounded subset. The choice of  $V$ and boundedness essentially comes from previous lemma with the aid of cleverly forming  a partition of unity  and invoking the result of Bochner and Phillips (Theorem \ref{kb}).

\begin{lemma} \label{bdtobd}
Let $\cX_1$ be a unital subalgebra of a unital Banach algebra $\cX$ both having the same unit $1_\cX$, $E \subset \cX_1$ be such that it contains the unit ball of $\cX_1$, and let $F : E \to \cX$ be a function such that $T_F$ maps $A^1_\omega(\Gamma,\cX_1)$ to $A^q(\Gamma,\cX)$. Then $T_F$ maps a certain neighborhood of $0$ in $A^1_\omega(\Gamma,E)$ to a bounded subset of $A^q(\Gamma,\cX)$.
\end{lemma}
\begin{proof}
By Lemma \ref{lbd} and translation invariance of norms $\|\cdot\|_{A^1_\omega}$ and $\|\cdot\|_{A^q}$ (Remark \ref{ter}), there is a neighborhood $V$ of 0 in $\Gamma$, $L>0$ and $\eta>0$ such that $\|F(f)\|_{A^q} < L$ for all $f\in A^1_{\omega}(\Gamma,E)$ with $\|f\|_{A^1_\omega}<\eta$ and support of $f$ in some translate of $V$. Let $U$ and $W$ be two neighborhoods of 0 in $\Gamma$ such that $\overline{W} \subset U \subset \overline{U} \subset V.$  By Theorem \ref{Domar}, choose $\widetilde \alpha,\widetilde \beta \in A^1_\omega(\Gamma,\mathbb{C})$ such that $0 \leq \widetilde \alpha \leq1$, $0 \leq \widetilde \beta \leq 1$, $\widetilde \alpha_{|W}=1$, $\supp(\widetilde \alpha)\subset U$, $\widetilde \beta_{|U}=1$ and $\supp(\widetilde \beta)\subset V,$ and consider the functions $\alpha=1_\cX\widetilde \alpha$ and $\beta=1_\cX\widetilde \beta$ in $A^1_\omega(\Gamma,E)$. 
Since $\Gamma$ is compact, finite translations of $W$, say $W_1,W_2,\dots,W_n$, cover $\Gamma$. Let $\alpha_i$ and $\beta_i$, for $1\leq i \leq n$, be the corresponding translates of $\alpha$ and $\beta$, and let $$\phi_i=\frac{\alpha_i}{\alpha_1+\alpha_2+\dots+\alpha_n} \quad (1\leq i \leq n).$$ Then $\widetilde\alpha_1(\gamma)+\widetilde\alpha_2(\gamma)+\dots+\widetilde\alpha_n(\gamma)\neq0$ for all $\gamma\in\Gamma$ implies that $\phi_i \in A^1_\omega(\Gamma,E)$ for each $1 \leq i\leq n$ by Theorem \ref{kb}. In addition, $\sum_{i=1}^n \phi_i \equiv 1_\cX$, that is, $\{\phi_i\}_{i=1}^n$ forms a partition of unity.

Let $f\in A^1_\omega(\Gamma,E)$ be such that $\|f\|_{A^1_\omega}<\frac{\eta}{\|\beta\|_{A^1_\omega}}$. Then, for $1\leq i \leq n$, $\beta_i f$ has support in some translate of $V$ and $\|\beta_i f\|_{A^1_\omega} \leq \|\beta_i \|_{A^1_\omega} \|f\|_{A^1_\omega} < \eta$. This implies $\|F(\beta_i f) \|_{A^q} <L$. Since  $\beta_i$ is $1_\cX$ on support of $\alpha_i$ and hence $\phi_i$, we have $$F(f) = \sum_{i=1}^n \phi_i F(f) = \sum_{i=1}^n \phi_i F(\beta_i f).$$ This along with \eqref{Holder relation} gives $$ \|F(f) \|_{A^q} < L \sum_{i=1}^n \|\phi_i \|_{A^1_\omega}. $$ Let $\delta=\frac{\eta}{\|\beta\|_{A^1_\omega}}$ and $C=L \sum_{i=1}^n \|\phi_i \|_{A^1_\omega}$. Then $\|F(f) \|_{A^q} < C$ for all $f\in A^1_\omega(\Gamma,E)$ with $\|f\|_{A^1_\omega}<\delta$. This concludes the proof.
\end{proof}

\begin{remark}
In  Lemmas \ref{lbd} and \ref{bdtobd}, we have taken $E \subset \cX_1$ such that it contains the unit ball of $\cX_1$ to invoke Theorem \ref{Domar}. But we way replace the unit ball by even a smaller radius ball and use a modified version of Theorem \ref{Domar} giving the desired results.
\end{remark}

\begin{remark}\label{whyq<2} In Lemmas \ref{lbd} and \ref{bdtobd}, the hypothesis  $q<2$ is not required. However, in  Lemma \ref{sup}  and in other proofs, Plancheral theorem will be utilized along with the fact that $\ell^q(G)\subset\ell^2(G)$ and so we have taken $1\leq q <2$.
\end{remark}


The result of  Rudin \cite[Theorem 2]{rud}  states that supremum of $\|e^{if}\|_{A^q}$ is greater than $K_q^n$ where $K_q>1$ is constant and the supremum is taken over all real valued $f\in A^1(\Gamma)$ with norm $\leq n$. Below Lemma \ref{sup} shows that the same result even holds for smaller algebra $A^1_\omega(\Gamma)$ (weighted) which is a subalgebra of $A^1(\Gamma)$ and in the  vector valued setting. The modifications has to be done in accordance to the weight.  No particular property of weight is used. The main idea is to just construct function $g_x$ in \eqref{dgl} in accordance with the weight. It plays an important part in the proof of the main theorem as it will give the convergence of power series form of $F$.

\begin{lemma} \label{sup}
Let $\cX$ be a unital commutative Banach algebra, and let $\omega$ be a (not necessarily admissible) weight on $G$. For $n\in\mathbb{N}$, let $S_n(\Gamma,\cX)$ be the set of all trigonometric polynomials $f \in A^1_{\omega}(\Gamma,\cX)$ such that $\|f\|_{A^1_\omega}\leq n$. Then there is a constant $K_q>1$ such that $$\sup_{f\in S_n(\Gamma,\cX)} \|e^{if}\|_{A^q}\geq K_q^n.$$
\end{lemma}
\begin{proof} 
We first show it for $\cX=\mathbb{C}$. For $x_0\in G,$ keeping up with our notations as in Remark \ref{notation}, define $\widehat{\delta_{x_0}}\in \ell^1_\omega(G,\mathbb{C})$ by
\begin{equation}\label{sme}
      \widehat{\delta_{x_0}}(x)= \begin{cases} 1 & if \ x=x_0\\
      0 & if \ x\neq x_0
    \end{cases}.
\end{equation}
Then $\|\widehat{\delta_{x_0}}\|_{\ell^1_\omega}=\omega(x_0)$ and its Fourier transform $\delta_{x_0}\in A^1_\omega(\Gamma,\mathbb{C})$ is given by 
$$\delta_{x_0}(\gamma) = \sum_{x\in G} \widehat{\delta_{x_0}}(x)\gamma(-x) = \gamma(-x_0) \quad (\gamma\in\Gamma).$$

Next, for $x_0\in G$, define $g_{x_0}\in A^1_\omega(\Gamma,\mathbb{C})$ by 
\begin{align}\label{dgl} 
    g_{x_0}(\gamma)=\exp\left(\frac{i 2\Re(\gamma(x_0))}{\omega(x_0)+\omega(-x_0)}\right) \quad (\gamma\in\Gamma),
\end{align} where $\Re(z)$ denotes the real part of a complex number $z$. First, we show that there is a constant $K_q>1$ independent of $x_0$ such that 
    \begin{equation} \label{eq:rud 2}
        \|g_{x_0}\|_{A^q}\geq K_q.
    \end{equation}
In view of   \eqref{invFT}, \eqref{sme} and \eqref{dgl},  we have
\begin{align*}
    \widehat{g_{x_0}}(x) 
    &= \int_\Gamma g_{x_0}(\gamma) \gamma(x) d\gamma \\
    &= \int_\Gamma \exp\left(\frac{i 2\Re(\gamma(x_0))}{\omega(x_0)+\omega(-x_0)}\right) \gamma(x) \ d\gamma\\
    &= \int_\Gamma \exp\left(\frac{i(\gamma(x_0)+\gamma(-x_0))}{\omega(x_0)+\omega(-x_0)}\right) \gamma(x) \ d\gamma\\
    &= \int_\Gamma \sum_{n\in\mathbb{N}\cup\{0\}} \frac{i^n}{n!} \frac{(\gamma(x_0)+\gamma(-x_0))^n}{(\omega(x_0)+\omega(-x_0))^n} \gamma(x) \ d\gamma\\
    &= \int_\Gamma \left[ \sum_{n\in\mathbb{N}\cup\{0\}} \frac{i^n}{n!} \frac{1}{(\omega(x_0)+\omega(-x_0))^n}  \sum_{k=0}^n \binom{n}{k} \gamma(x_0)^{n-k} \gamma(-x_0)^k \right] \gamma(x) \ d\gamma\\
    &= \sum_{n\in\mathbb{N}\cup\{0\}} \left[ \frac{i^n}{n!} \frac{1}{(\omega(x_0) + \omega(-x_0))^n} \sum_{k=0}^n \binom{n}{k} \int_\Gamma  \gamma((n-k)x_0-kx_0) \gamma(x) \ d\gamma \right]\\
    &= \sum_{n\in\mathbb{N}\cup\{0\}} \left[ \frac{i^n}{n!(\omega(x_0) + \omega(-x_0))^n} \sum_{k=0}^n \binom{n}{k} \widehat{\delta_{(2k-n)x_0}}(x) \right],
\end{align*}
where the last relation is due to the fact that $\int_\Gamma \gamma(x) d\gamma$ is $1$ only if $x=0$ and zero otherwise. This shows that  $\widehat{g_{x_0}}$ is supported in the cyclic subgroup of $G$ generated by $x_0$. Hence, say $$\|g_{x_0}\|_{A^q}=K(q,n)  $$ depends only on the order $n$ of $x_0$ in $G$ and $q$. As $|g_{x_0}(\gamma)|=1$ for all $\gamma\in\Gamma$, $\|g_{x_0}\|_{L^2(\Gamma)}=1$ and so by the Plancherel theorem $\|g_{x_0}\|_{A^2(\Gamma)}=1$. For $x\in G$ with $\widehat{g_{x_0}}(x) \neq 0$, $|\widehat{g_{x_0}}(x)|<1$ and this implies that $|\widehat{g_{x_0}}(x)|^2<|\widehat{g_{x_0}}(x)|^q$. Therefore, $K(q,n)>1$ for all $n$. Thus, defining $K_q=\inf_n K(q,n)>1$, we get \eqref{eq:rud 2}.

Observe that if for each $\mu<1$ and $n\in\mathbb{N}$, there are $x_1,x_2,\dots,x_n\in G$ such that 
\begin{equation} \label{eq:rud 3}
    \|g_{x_1}g_{x_2} \dots g_{x_n}\|_{A^q} > (\mu K_q)^n,
\end{equation}
then the lemma follows by taking    $$f(\gamma)=\frac{2\Re(\gamma(x_1))}{\omega(x_1)+\omega(-x_1)}+\frac{2\Re(\gamma(x_2))}{\omega(x_2)+\omega(-x_2)}+ \dots + \frac{2\Re(\gamma(x_n))}{\omega(x_n)+\omega(-x_n)} $$ which is in $S_n(\Gamma,\mathbb{C})$ as 
\begin{align*}
    \widehat{f}=\sum_{i=1}^n \frac{\delta_{x_i}+\delta_{-x_i}}{\omega(x_i)+\omega(-x_i)} \quad \text{and} \quad
    \|f\|_{A^1_\omega}=\sum_{x\in G} \left| \sum_{i=1}^n \frac{\delta_{x_i}(x)+\delta_{-x_i}(x)}{\omega(x_i)+\omega(-x_i)} \right| \omega(x) = n.
\end{align*}
Also, by \eqref{eq:rud 2}, \eqref{eq:rud 3} follows from 
\begin{equation} \label{eq:rud 4}
    \|g_{x_1}g_{x_2} \dots g_{x_n}\|_{A^q} > \mu^n \|g_{x_1}\|_{A^q} \|g_{x_2}\|_{A^q} \dots \|g_{x_n}\|_{A^q}.
\end{equation}
So, it is sufficient to show \eqref{eq:rud 4} which will be achieved by principle of induction. Clearly, it is true for $n=1$. Now, assume that it is true for some $n\in\mathbb{N}$ and let $f=g_{x_1}g_{x_2}\dots g_{x_n}$. It remains to show that there is $x_{n+1}\in G$ such that 
\begin{equation} \label{eq:rud 5}
    \|f g_{x_{n+1}}\|_{A^q} > \mu \|f\|_{A^q} \|g_{x_{n+1}}\|_{A^q}.
\end{equation}

Since $f\in A^1_\omega(\Gamma,\cX)$, it can be approximated by the partial sum of its Fourier series in the norms $\|\cdot\|_{A^1_\omega}$ and $\|\cdot\|_{A^q}$. 
So, it is sufficient to show \eqref{eq:rud 5} with $f$ as a trigonometric polynomial. Then the support of $\widehat{f}$, say $H_1$,  will be a finite subset of $G$. The remaining proof is based on the fact that if $h\in A^q(\Gamma,\cX)$ with $H_2\subset G$ as the support of $\widehat{h}$, and if no $x\in G$ has more than one representation of the form $x=y_1+y_2$ for $y_1\in H_1, y_2\in H_2$, then $\|fh\|_{A^q}=\|f\|_{A^q} \|h\|_{A^q}$. We shall consider three cases.

Case I: If $x_0\in G$ is of infinite order, then there is $m\in\mathbb{N}$ such that the cyclic group $H$ generated by $x_{n+1}=mx_0$ satisfies the above condition and so 
\begin{equation} \label{eq:rud 6}
    \|fg_{x_{n+1}}\|_{A^q}=\|f\|_{A^q} \|g_{x_{n+1}}\|_{A^q}.
\end{equation}

Case II: If $G$ is of bounded order, then $G$ is a direct sum of finite cyclic groups. So, there is $x_{n+1}\in G$ such that the cyclic group $H$ generated by $x_{n+1}$ has only 0 in common with $H_1$, support of $f$, and \eqref{eq:rud 6} follows.

Case III: The only remaining case is where $G$ is not of bounded order and every $x\in G$ has finite order. This means the group $E$ generated by $H_1$ is finite. For each $N\in\mathbb{N}$, there is $x_{n+1}$ such that $mx_{n+1}\notin E$ if $|m|<2N$ or else $G$ is of bounded order. Let 
\begin{equation} \label{eq:rud 7}
    g_{x_{n+1}}(\gamma)=\sum_{m\in\mathbb{Z}} a_m \gamma(mx_{n+1}) \quad \text{with} \quad \sum_{m\in\mathbb{Z}} |a_m| <\infty.
\end{equation}
Let $h(\gamma)$ be the partial sums of \eqref{eq:rud 7} with $|m|<N$. Then $\|fg\|_{A^q} = \|f\|_{A^q} \|g\|_{A^q}$. We can make $\|g_{x_{n+1}}-h\|_{A^q}$ as small as we want by taking $N$ large enough, and thus \eqref{eq:rud 5} follows. This completes the proof for $\cX=\mathbb{C}$.

Now, if $\cX$ is any unital commutative Banach algebra, then $A^1_\omega(\Gamma,\mathbb{C}) \subset A^1_\omega(\Gamma,\cX)$  implies $S_n(\Gamma,\mathbb{C})\subset S_n(\Gamma,\cX)$ and thus $$\sup_{f\in S_n(\Gamma,\cX)} \|e^{if}\|_{A^q} \geq \sup_{f\in S_n(\Gamma,\mathbb{C})} \|e^{if}\|_{A^q} \geq K_q^n.$$
\end{proof}

Lastly, we state a lemma required to show the continuity of $F$.
\begin{lemma} \label{lem:cont}
Let $\cX$ be a commutative unital Banach algebra, and let $\{f_n\}_{n\geq1}$ be a sequence in $A^q(\Gamma,\cX)$ such that $\|f_n(\gamma)\|\leq M_1$ and $\|f_n\|_{A^q}\leq M_2$ for some positive constants $M_1,M_2$ and for all $n\in\mathbb{N}$ and $\gamma\in\Gamma$. If $f_n(\gamma)\to f(\gamma)$ a.e. in $\Gamma$, then $\|f\|_{A^q}\leq M_2$.   
\end{lemma}
\begin{proof}
    The proof follows from the Lebesgue's dominated convergence  theorem for Banach space valued functions.
\end{proof}

\section{Necessary condition} \label{sec:nec}
In this section, we shall show the necessary parts, that is, Theorems \ref{mbk} and \ref{tfam} \eqref{nc1}.

\subsection{Weighted Fourier algebra}
The strategy of the proofs here follows back to the work of Helson, Kahane, Katznelson and Rudin in \cite{helson, rb, rud}. Following this strategy, we shall prove the results for $I=\{(x_1,x_2)\in\mathbb{R}^2 : x_1, x_2 \in [-1,1] \}$ in parts \eqref{twovarq} and \eqref{twovar}, and $I=[-1,1]$ in parts \eqref{vecvalq} and \eqref{vecval} of Theorem \ref{mbk}.
\begin{proof}[\textbf{Proof of Theorem \ref{mbk} \eqref{twovarq}}]
Assume $F(0)=0$ and  take $$I=\{x=(x_1,x_2)\in\mathbb{C}^2 : x_1, x_2 \in [-1,1] \}.$$   Since $G$ is discrete, by Lemma \ref{bdtobd},  there are $\delta>0$ and $C>0$ such that  for $f\in A^1_\omega(\Gamma,\mathbb{R}^2)$ with range of $f$ in $I$, we have
  \begin{eqnarray}\label{bss1}
     \|F(f)\|_{A^q}<C \quad \text{whenever} \quad  \|f\|_{A^1_\omega}<\delta.
  \end{eqnarray}
We claim that   $F$ is continuous. If not, suppose that   $F$ is not continuous at the $0$. So, there is a sequence $\{t_n\} \subset I$ such that $t_n\to 0$ but $F(t_n)\to l\neq 0.$  For $t\in I,$ define  constant functions on $\Gamma$ as follows
$$g_t(\gamma)=t \quad (\gamma\in\Gamma).$$
Clearly $\|g_t\|_{A^1_\omega} = |t| \omega (0).$ By  \eqref{bss1}, if $|t|<\frac{\delta}{\omega(0)}$, then $\|F(g_t)\|_{A^q}<C$ and so for all $\gamma\in\Gamma$, 
\begin{align} \label{F(t)<C}
    |F(g_t)(\gamma)| = |F(t)| = \|F(g_t)\|_{A^q} < C
\end{align}
as $g_t$ is constant implies $F(g_t)$ is also constant.
Decompose   $\Gamma= E\cup E^{c},$  with  $E$ is  closed. Let $\{W_n\}$ be a sequence of closed sets in $E^c$ such that $W_n\subset W_{n+1}$ and measure of $E\cup W_n$ goes to 1 as $n\to\infty$. Let $\{f_n\}_{n=1}^{\infty} \subset A^1_\omega(\Gamma,\mathbb{R})$ be such that 
\begin{equation*}
    f_n=\begin{cases} 1 & on \ E\\
    0 & on \ W_n    
    \end{cases}.
\end{equation*}
Given $k\in\mathbb{N}$, there is $n_k$ such that $|t_{n_k}|\|f_k\|_{A^1_\omega}<\delta$ and so by \eqref{bss1},  $\|F(t_{n_k}f_k)\|_{A^q}<C$. By  the above choice of $f_n$ and \eqref{F(t)<C}, we get $F(t_{n_k}f_k)(\gamma)\rightarrow l \chi_{_E}(\gamma)$ a.e. and  $|F(t_{n_k}f_k)(\gamma)|<C$ for all $k\in\mathbb{N}$.  So, by Lemma \ref{lem:cont},  we get $\|\chi_{_E}\|_{A^q}<\frac{C}{|l|}$. Since $E$ can be any closed subset of $\Gamma$, it contradicts Theorem \ref{thm:rud 4}. So, $F$ is continuous at $0$ and consequently in $I$ by considering $H(x)=F(x_0+x)-F(x_0)$ instead of $F(x)$ for $x_0\in I$. In fact, by arguments as above it can be shown that $H$ is continuous at $0$ and it implies that $F$ is continuous at $x_0$. Hence, the claim is true.

Fix $0<r<\min\left\{\frac{\delta}{2e},\frac{1}{2}\right\}$, and define $F_1:\mathbb{R}^2\to \mathbb{C}^2$ by 
\begin{align} \label{def:F_1}
    F_1(s)=F(r\sin(s))=F(r\sin s_1,r \sin s_2) \quad (s=(s_1,s_2)\in\mathbb{R}^2).
\end{align} 
If $f\in A^1_\omega(\Gamma,\mathbb{R}^2)$, $\|f\|_{A^1_\omega}<1$ and if $a=(a_1,a_2)\in\mathbb{R}^2$, then 
\begin{align} \nonumber
\|\sin(f+a)\|_{A^1_\omega} 
&\leq |\cos(a)| \|\sin(f)\|_{A^1_\omega} + |\sin(a)| \|\cos(f)\|_{A^1_\omega} \\ \nonumber
&= |(\cos a_1,\cos a_2)| \sum_{n\in\mathbb{N}} \frac{1}{(2n-1)!} \|f^{2n-1}\|_{A^1_\omega} + |(\sin a_1,\sin a_2)| \sum_{n\in\mathbb{N}} \frac{1}{(2n)!} \|f^{2n}\|_{A^1_\omega} \\ \nonumber
&\leq 2e^{\|f\|_{A^1_\omega}} \\ 
&\leq 2e 
\end{align} 
and so $\|r\sin(f+a)\|_{A^1_\omega}\leq 2re < \delta$. By Lemma \ref{bdtobd}, that is, \eqref{bss1}, 
\begin{equation}\label{eqn:Aqnorm<C}
    \|F_1(f+a)\|_{A^q}=\|F(r\sin(f+a))\|_{A^q}<C.
\end{equation}
The continuity of $F$ implies that $F_1$ is a continuous function on $\mathbb{T}^2$. So, $F_1$ has a Fourier series representation of the form
\begin{align}\label{fsf}
F_1(s_1,s_2) \sim \sum_{m,n\in\mathbb{Z}} c_{mn} e^{ims_1} e^{ins_2} \quad (s_1,s_2\in\mathbb{T}).    
\end{align}

Let $f\in A^1_\omega(\Gamma,\mathbb{R}^2)$ be a trigonometric polynomial such that $\|f\|_{A^1_\omega}<1$. Then $\|F_1(f+a)\|_{A^q}<C$ by \eqref{eqn:Aqnorm<C}.    For all $m,n\in\mathbb{Z}$  and $\gamma\in\Gamma$,
\begin{align*} 
c_{mn} e^{i(m,n) f(\gamma)} 
&= \frac{1}{(2\pi)^2} \int_0^{2\pi} \int_0^{2\pi} F_1(f(\gamma) + (t_1,t_2)) e^{-i(m,n)(t_1,t_2)} dt_1 dt_2 \\ 
&= \lim_{\substack{M,N\to\infty}} \frac{1}{MN} \sum_{j=1}^M \sum_{k=1}^N F_1\left( f(\gamma) + \left(\frac{2\pi j}{M}, \frac{2\pi k}{N} \right) \right) e^{-i\left(\frac{2\pi j}{M} + \frac{2\pi k}{N} \right)}. \end{align*} 
So, from Lemma \ref{lem:cont},  it follows that $$|c_{mn}| \|e^{i(m,n) f}\|_{A^q}\leq C. $$
And by Lemma \ref{sup} , taking supremum over all such $f\in A^1_{\omega}(\Gamma,\mathbb{R}^2)$, we have 
\begin{eqnarray}\label{fsfdecay}
    |c_{mn}|< C K_q^{-|(m,n)|} < C K_q^{-|m|-|n|}
\end{eqnarray}
 for some $K_q>1$ and all $m,n\in\mathbb{Z}$. Thus, the series $$\sum_{m,n\in\mathbb{Z}} c_{mn} e^{im(s_1+it_1)} e^{in(s_2+it_2)}$$ converges absolutely for all $t_1,t_2\in(-\log K_q,\log K_q)$.

It means that $$\widetilde{F_1}(s_1 + i t_1,s_2 + i t_2) =\sum_{m,n\in\mathbb{Z}} c_{mn} e^{im(s_1+it_1)} e^{in(s_2+it_2)} $$ defines a holomorphic function on $\{(s_1 + i t_1,s_2 + i t_2)\in\mathbb{C}^2:s_1,s_2\in\mathbb{R}, t_1,t_2\in(-\log K_q,\log K_q)\}$. Since $F_1(s)=F(r\sin(s))$ ($s\in\mathbb{R}^2$) is the restriction of $\widetilde{F_1}$ to $\mathbb{R}^2$, $F_1$ is real analytic  on $\mathbb R^2$, and it implies that $F$ is analytic in a neighborhood of (0,0) in $I$. 
Now, let $0\neq x_0$ be an interior point of $I$. Consider the translation of $F$ given by $F_{x_0}(x)=F(x+x_0)$. Then it can be shown similarly that $F_{x_0}$ is real analytic at $0$ which in turns implies that $F$ is real analytic at $x_0$.

Next, consider the boundary point $(s_0,t_0)\in I$ with $s_0=1$ and $|t_0|<1$, and consider the function $$F_2(s,t)=F(1-s^2,t^2) \quad (s\in[-1,1], t\in(-1,1)).$$ Then $F_2$ operates from $A^1_\omega(\Gamma,\mathbb{C}^2)$ to $A^q(\Gamma,\mathbb{C}^2)$ and since $F_2$ is an even function, for some $\epsilon>0$, $$F_2(s,t)=\sum_{m,n\in\mathbb{N}_0} c_{mn} s^{2m} t^{2n} \quad (s,t\in(-\epsilon,\epsilon)).$$
This gives $F(1-s,t)=\sum_{m,n\in\mathbb{N}_0} c_{mn} s^m t^n$ for $s,t\in[0,\epsilon^2)$ and so $F$ is analytic at $(s_0,t_0)$ by translation in second variable. The analyticity of $F$ at the remaining boundary points can be shown similarly. 

Lastly, if $I$ is any subset of $\mathbb{C}^2$, then take $$F_1(s_1,t_1,s_2,t_2)=F(r\sin(s_1),r\sin(t_1),r \sin(s_2),r \sin(t_2))$$ which will give the Fourier series 
$$F_1(s_1,t_1,s_2,t_2)=\sum_{m_1,m_2,n_1,n_2\in\mathbb{Z}} c_{m_1 m_2 n_1 n_2} e^{i(m_1 s_1 + m_2 t_1)} e^{i(n_1 s_2 + n_2 t_2)}$$
and the coefficients $c_{m_1 m_2 n_1 n_2}$ can be estimated as above.
\end{proof}

\begin{proof}[\textbf{Proof of Theorem \ref{mbk} \eqref{twovar}}]
Here we have $1\leq q <2$ and $\omega$ is a $q-$algebra weight. If $\Gamma$ is compact, then it follows from Theorem \ref{mbk} \eqref{twovarq}. Now assume that $\Gamma$ is not compact. If $\Gamma$ has an infinite compact subgroup, say $\Lambda$, and if $T_F$ acts on $A^q_{\omega}(\Gamma,\mathbb{C}^2)$, then $T_F$ acts on $A^q_{\omega}(\Lambda,\mathbb{C}^2)$   by an application of \cite[Theorem 2.7.4]{rb} which states that every function in $A^q_{\omega}(\Lambda,\mathbb{C}^2)$ is a restriction of a function in $A^q_{\omega}(\Gamma,\mathbb{C}^2)$. Thus, $F$ is real analytic by above arguments.

Now, if every compact subgroup of $\Gamma$ is finite, then the fact that $\Gamma$ is a non-discrete group (as $G$ is not compact) and the structure theorem \cite[Theorem 2.4.1]{rb} implies that $\Gamma$ has a closed subgroup which is isomorphic to $\mathbb{R}$. This means that $T_F$ acts on $A^q_{\omega}(\mathbb{R},\mathbb{C}^2)$ and thus on $A^q_{\omega}(\mathbb{T},\mathbb{C}^2)$. This completes the proof as $\mathbb{T}$ is compact.
\end{proof}

\begin{proof}[\textbf{Proof of Theorem \ref{mbk} \eqref{vecvalq} and \eqref{vecval}}]
The proof essentially follows from the same techniques as in Theorem \ref{mbk} \eqref{twovarq} and Theorem \ref{mbk} \eqref{twovar} along with the fact that $\mathbb{C}$ can be identified with a subalgebra of $\cX$, both having the same unit, through the identification $\alpha\in\mathbb{C} \to \alpha 1_\cX \in \cX$.
\end{proof}

Next we state a few corollaries of above results.

\begin{corollary}
For $1\leq q_1 \leq q_2 <2$, compact $\Gamma$, and admissible ($q_i-$algebra, if $q_i>1$) weights $\omega_i$ on $G$ for $i=1,2$ with $\omega_2 \lesssim \omega_1$, if $T_F$ maps $A^{q_1}_{\omega_1}(\Gamma,\mathcal Y)$ to $A^{q_2}_{\omega_2}(\Gamma,\cX)$ for $\cX=\mathcal Y=\mathbb{C}^2$ or $\mathcal Y=\mathbb{C}$ and $\cX$ a unital commutative Banach agebra, then $F$ is real analytic in the domain of $F$.
\end{corollary}
\begin{proof}
    From \eqref{ineqnormsdis} in Lemma \ref{lem:weighted Fourier algebras inclusions}, it follows that $T_F$ maps $A^{1}_{\omega_1}(\Gamma,\mathcal Y)$ to $A^{q_2}(\Gamma,\cX)$. Thus the proof follows from Theorems \ref{mbk} \eqref{twovarq} and \eqref{vecvalq}.
\end{proof}

\begin{corollary}
For $1\leq q <2$, non-compact $\Gamma$, and admissible ($q-$algebra, if $q>1$) weights $\omega_i$ on $G$ for $i=1,2$ with $\omega_2 \lesssim \omega_1$, if $T_F$ maps $A^{q}_{\omega_1}(\Gamma,\mathcal Y)$ to $A^{q}_{\omega_2}(\Gamma,\cX)$ for $\cX=\mathcal Y=\mathbb{C}^2$ or $\mathcal Y=\mathbb{C}$ and $\cX$ a unital commutative Banach agebra, then $F$ is real analytic in the domain of $F$.
\end{corollary}

\subsection{Modulation and amalgam spaces}
Here we shall prove Theorem \ref{tfam} \eqref{nc1}: if the composition operator $T_F$ maps $\widehat{w}_\omega^{p,1}$ to $\widehat{w}^{p,q}$, then, necessarily, $F$ is real analytic on $\mathbb{R}^{2}.$ And also a similar necessity condition for modulation and Wiener amalgam spaces. We start with the following:

Recall that for $1\leq q<\infty$ and a weight $\omega$ on $\mathbb Z^d$, $A^q_\omega(\mathbb T^{d})$ is the class of all complex functions $f$ on the $d-$torus $\mathbb T^d$ whose Fourier coefficients
$$\widehat{f}(m)=\int_{\mathbb T^d}f(x)e^{-2\pi i m\cdot x} dx,   \ (m\in \mathbb Z^{d})$$
satisfy the condition
$$\|f\|_{A^{q}_\omega(\mathbb T^{d})}:=\|\widehat{f}\|_{\ell^{q}_\omega}<\infty.$$

We now define the local-in-time versions of the  Fourier amalgam, Wiener amalgam  and modulation spaces  in the following way. Given an interval $I=[0,1)^{d}$, let 
\begin{enumerate}
    \item $\widehat{w}_\omega^{p,q}(I)$ be the restriction of $\widehat{w}_\omega^{p,q}$ onto $I$ via 
\begin{eqnarray}\label{litwd}
\|f\|_{\widehat{w}_\omega^{p,q}(I)}:=\inf \{\|g\|_{\widehat{w}_\omega^{p,q}}:g=f \ \text{on} \ I \}.
\end{eqnarray}
    \item $W_\omega^{p,q}(I)$ be the restriction of $W_\omega^{p,q}$ onto $I$ via 
\begin{eqnarray}
\|f\|_{W_\omega^{p,q}(I)}:= \inf \{\|g\|_{W_\omega^{p,q}}:g=f \ \text{on} \  I \}.
\end{eqnarray}
    \item $M_\omega^{p,q}(I)$ be the restriction of $M_\omega^{p,q}$ onto $I$ via
\begin{eqnarray}
\|f\|_{M_\omega^{p,q}(I)}=\inf \{\|g\|_{M_\omega^{p,q}}:g=f \ \text{on} \  I \}.
\end{eqnarray}
\end{enumerate}

Now we introduce periodic modulation and amalgam spaces. Using the identification of $I=[0,1)^d$ with $\mathbb{T}^d$ via the mapping $(x_1,x_2,\dots,x_d)\in I \mapsto (e^{2\pi i x_1},e^{2\pi i x_2},\dots,e^{2\pi i x_d})\in\mathbb{T}^d$, we define 
\begin{align} \label{def:periodicspaces}
    \widehat{w}_\omega^{p,q}(\mathbb{T}^d):=\widehat{w}_\omega^{p,q}(I), \quad W_\omega^{p,q}(\mathbb{T}^d):=W_\omega^{p,q}(I)\quad \text{and} \quad M_\omega^{p,q}(\mathbb{T}^d):=M_\omega^{p,q}(I).
\end{align}

The next result ensures that the modulation, Wiener amalgam and Fourier amalgam spaces coincide with the classical weighted Fourier algebra. Specifically, we have

\begin{proposition}[see Lemma 1 in \cite{ko}, Section 5 in \cite{moT} and  \cite{forlano2020deterministic,BhiStrongIll}]  \label{B1} 
For $1\leq p , q \leq \infty$,
$$\widehat{w}^{p,q}(\mathbb T^{d}) = W^{p,q}(\mathbb{T}^d) = M^{p,q}(\mathbb{T}^d) = A^{q}(\mathbb T^{d}),$$ 
with norm inequality
$$\|f\|_{\widehat{w}^{p,q}(\mathbb T^d)} \asymp \|f\|_{W^{p,q}(\mathbb{T}^d)} \asymp \|f\|_{M^{p,q}(\mathbb T^d)} \asymp \|f\|_{A^{q}(\mathbb T^d)}.$$
\end{proposition}
The detailed proof  of Proposition  \ref{B1} for the modulation and Wiener amalgam spaces is provided in  \cite[Lemma 1]{ko} and \cite[Section 5]{moT}. While  for the Fourier amalgam spaces, we shall  provide the proof in  Appendix \ref{apd} (Proposition \ref{faT}).

\begin{proposition} \label{promaps}
Let $1\leq p \leq \infty$, $1\leq q <2$,  $\omega$ be an admissible weight on $\mathbb{Z}^d$, and $$(X,Y)=(M_{\omega}^{p,1},M^{p,q}) \ \text{or} \ (W_{\omega}^{p,1},W^{p,q}) \ \text{or} \ (\widehat{w}_{\omega}^{p,1},\widehat{w}^{p,q}).$$ Suppose that $T_{F}$ is the composition operator associated to a complex function $F$ on $\mathbb C$. If $T_{F}$ maps $X$ to $Y$, then $T_{F}$ maps  $A_{\omega}^{1}(\mathbb T^{d})$ to $A^q(\mathbb T^{d}).$
\end{proposition}
\begin{proof}
First, we consider the case $(X,Y)=(\widehat{w}_{\omega}^{p,1},\widehat{w}^{p,q}).$ Let $ f\in A^1_\omega(\mathbb{T}^d).$ Then $f^{\ast}(x)=f(e^{2 \pi ix_1},..., e^{2 \pi ix_d})$ is a periodic function on $\mathbb R^d$ with absolutely convergent Fourier series $$f^*(x)=  \sum_{m \in \mathbb Z^d} \widehat{f} (m)\, e^{2 \pi i m \cdot x} .$$ 
Choose $g \in \widehat{w}^{p,1}_\omega(\mathbb R^d)$ such that $g\equiv1$ on $I=[0, 1)^d.$ Then we claim that
$gf^{\ast}\in   \widehat{w}_\omega^{p,1}.$ Once the claim is assumed, by hypothesis, we have
\begin{eqnarray}\label{rh}
F(g  f^{\ast})\in   \widehat{w}^{p,q}.
\end{eqnarray}
Note that if $z\in \mathbb T^d$, then $z= (e^{2\pi ix_1},... , e^{2\pi ix_d})$ with some unique $x=(x_1,... ,x_d) \in I$, hence 
\bea\label{cnct} F(f(z))=F( f^*(x))= F( (g  f^{\ast}) (x)).\eea
Now if $\phi\in C_c^\infty(\mathbb T^d)$, then $ g \phi^*$ is a compactly supported function on $\R^d$. Also $\phi(z) = g(x) \phi^*(x)$ for every $x \in I$, as per the notation above and hence 
\bea \label{3.2} \phi(z) F(f)(z)= g(x)\phi^*(x) F(gf^*)(x),\eea for some $ x \in I$.

By \eqref{3.2}, Proposition \ref{B1}, \eqref{def:periodicspaces}, \eqref{litwd} and Proposition \ref{proalgebra} \eqref{faminc},   
we obtain
\begin{eqnarray}
\|\phi F(f)\|_{ A^{q}(\mathbb T^d)}
& = & \| g \phi^* F(gf^{\ast})\|_{ A^{q}(\mathbb T^d)}\nonumber \\
& \asymp & \|g\phi^* F(gf^*)\|_{\widehat{w}^{p,q}(\mathbb T^{d})} \nonumber\\
& = & \|g\phi^* F(gf^*)\|_{\widehat{w}^{p,q} (I)}  \nonumber \\
& \lesssim & \|g \phi^{\ast} F(gf^*)\|_{\widehat{w}^{p,q}} \nonumber \\
& \lesssim & \|g\phi^*\|_{A^1} \|F(gf^*)\|_{\widehat{w}^{p,q}},\nonumber 
\end{eqnarray}
which is finite for every smooth cutoff function $\phi$ supported on $I$ by \eqref{rh}.  Now, by the compactness of $\T^d$, a partition of unity argument shows that $F(f) \in A^{q}(\mathbb T^d)$.

So, it just remains to prove the claim. For that, let $\mu=\sum_{k\in \mathbb Z^{d}}c_{k}\delta_{k}$, where $c_{k}=\widehat{f}(k)$ and  $\delta_{k}$ is the unit Dirac mass at $k.$  Then it is clear that $\mu$ is a complex Borel measure on $\mathbb R^{d},$ and the  total variation of $\mu, $ that is, $\|\mu\|= |\mu|(\mathbb R^{d})= \sum_{k\in \mathbb Z^{n}} |c_{k}|$ is finite. The Fourier-Stieltjes transform of $\mu$ is given by
\begin{eqnarray}
\widehat{\mu}(y)  
= \int_{\mathbb R^{d}} e^{-2\pi ix\cdot y} d\mu(x)
= \int_{\mathbb R^{d}} e^{-2\pi ix\cdot y} \left( \sum_{k\in \mathbb Z^{d}} c_{k}d\delta_{k}(x) \right)
= \sum_{k\in \mathbb Z^{d}} c_{k} \int_{\mathbb R^{d}} e^{-2\pi ix\cdot y} d\delta_{k} (x) 
= f^{\ast}(-y).\nonumber
\end{eqnarray}
So, $\widehat{f^\ast} = \mu = \sum_{m \in \mathbb Z^d} \widehat{f} (m) \, \delta_{m}.$ It follows that 
$$  \widehat{g} \ast \widehat{f^{\ast}}=  \sum_{m \in \mathbb Z^d} \widehat{f} (m)  \,  \widehat{g} \ast \delta_{m}.$$
Note that 
\begin{eqnarray*}
\|gf^*\|_{\widehat{w}^{p,1}_{\omega}}  =  \left\| \left\| \chi_{n+Q} \sum_{m \in \mathbb Z^d} \widehat{f} (m) \,  \widehat{g} \ast \delta_{m} \right\|_{L^p} \right\|_{\ell^{1}_{\omega}} \lesssim  \|f\|_{A^1_\omega(\mathbb T^d)} \|g\|_{\widehat{w}^{p,1}_{\omega}} < \infty.
\end{eqnarray*}
This proves the claim and hence the theorem for Fourier amalgam space.
Now, if $T_F$ maps $X$ to $Y$, whether $(X,Y)$ is either $(M_{\omega}^{p,1},M^{p,q})$ or  $(W_{\omega}^{p,1},W^{p,q})$, then using similar arguments as above and using Proposition \ref{proalgebra} \eqref{mwaminc}, the proof concludes.
\end{proof}

\begin{proof}[\textbf{Proof of Theorem \ref{tfam}}]
If $T_{F}$ maps $X$ to $Y$, then  $T_{F}$ maps  $A_{\omega}^{1}(\mathbb T^{d})$ to $A^q(\mathbb T^{d})$ by Proposition \ref{promaps}. Thus, the proof follows from Theorem \ref{mbk} \eqref{vecvalq} by taking $\Gamma=\mathbb{T}^d$, that is, $G=\mathbb{Z}^d$ and $\cX=\mathbb C$. Moreover, the condition $F(0)=0$ if $p<\infty$ follows from the fact that $\widehat{w}_\omega^{p,q} \subset \widehat{w}^{p,1} \subset \widehat{w}^{1,1} = A^1(\mathbb{R}^d)$ and by arguments similar to that given in \cite[Theorem 1.1 (i)]{db} for modulation and Wiener amalgam spaces.
\end{proof}

\section{Sufficient conditions} \label{sec:suff}
In this section, we shall prove Theorem \ref{suff} \eqref{ii} for weighted Fourier algebras on $\mathbb{R}^d$ and these results can be seen as the $q-$th power weighted continuous analogue of the Wiener-L\'evy theorem. The proof of Theorem \ref{suff} \eqref{i} follows using similar arguments. Then we shall show Theorem \ref{tfam} \eqref{sc} for modulation and amalgam spaces. We note that our approach of the proof is inspired by the weighted analogue of Wiener-L\'evy theorem in \cite{reiter}.

\subsection{Weighted Fourier algebra}
First, we show a result for all locally compact abelian group $G$ which gives a characterization for a function to be in $A^q_\omega(\Gamma,\cX)$. We begin with the following definition.

\begin{definition} \label{def:belongslocally}
Let $1\leq q <\infty$, $\omega$ be an admissible weight which is a $q-$algebra weight if $q>1$, and let $\phi:\Gamma \to \cX$ be a function. 
\begin{enumerate}
\item We say that $\phi$ belongs to $A^q_\omega(\Gamma,\cX)$ locally at point $\gamma_{0}\in \Gamma$ if there is a neighborhood $V$ of $\gamma_{0}$ and a function $h\in A^q_\omega(\Gamma,\cX)$ such that $\phi(\gamma)= h(\gamma)$ for every $\gamma \in V.$
\item If $\Gamma$ is not compact, then we say that $\phi$ belongs to $A^q_\omega(\Gamma,\cX)$ locally at $\infty$ if there is a compact set $K\subset \Gamma$ and  a function $h\in A^q_\omega(\Gamma,\cX)$ such that $\phi(\gamma)=h(\gamma)$ for every $\gamma \in \Gamma\setminus K$.
\end{enumerate} 
\end{definition}

The next lemma gives the useful criterion for functions to be in $A^q_\omega(\Gamma,\cX)$.

\begin{lemma} \label{lem:suff}
Let $1\leq q <\infty$, and let $\omega$ be an admissible weight which is a $q-$algebra weight if $q>1$. If $\phi$ belongs to $A^q_\omega(\Gamma,\cX)$ locally at every point of  $\Gamma$ (and at $\infty$ if $\Gamma$ is not compact), then $\phi \in A^q_\omega(\Gamma,\cX).$
\end{lemma}
\begin{proof}
Suppose that $\phi$ has a compact support $K$. By hypothesis, for each $\gamma \in K,$ there is a neighborhood $V_{\gamma}$ of $\gamma$ and $h_{\gamma} \in A^q_\omega(\Gamma,\cX)$ such that $\phi(\xi) = h_{\gamma}(\xi)$ for all $\xi\in V_{\gamma}.$ Since $\{V_{\gamma}: \gamma \in K \}$ forms an open cover of the compact set $K$, there are finite open sets, say $V_{\gamma_{1}}, V_{\gamma_{2}}, \dots, V_{\gamma_{n}}$, such that $K\subset \cup_{j=1}^{n} V_{\gamma_{j}}$. Let $h_{j}=h_{\gamma_j}\in A^q_\omega(\Gamma,\cX)$ for $1\leq j \leq n$. Then $\phi= h_{j}$ in $V_{\gamma_{j}}$ for $1\leq j\leq n$. Note that there exists \noindent
\begin{enumerate}
    \item open sets $W_{1},W_2,\dots, W_{n}$ with compact closures $\overline{W_{j}} \subset V_{\gamma_{j}}$ such that $K\subset \cup_{j=1}^{n}W_{j}$; and
    \item $k_{j}\in A^q_\omega(\Gamma,\cX)$ such that $k_{j}\equiv1_\cX$ on $\overline{W_{j}}$ and $k_{j}\equiv0$ outside $V_{\gamma_{j}}$ by Theorem \ref{Domar} and Remark \ref{Domarvector}.
\end{enumerate}
This implies that $\phi(\gamma) k_{j}(\gamma)=h_{j}(\gamma)k_{j}(\gamma)$ for all $\gamma\in\Gamma$ and $\phi k_{j}=h_{j} k_{j}\in A^q_\omega(\Gamma,\cX)$ for $1\leq j \leq n$ as $A^q_\omega(\Gamma,\cX)$ is an algebra. Now, if we  put
\begin{eqnarray}
\label{bs}
\psi=\phi\{1-(1-k_{1}) (1-k_{2})...(1-k_{n})\},
\end{eqnarray}
then $\psi\in A^q_\omega(\Gamma,\cX)$. The multiplier of $\psi$ in \eqref{bs} is 1 whenever one of $k_{i}$ is 1, and this happens at every point of $K$ and  $\phi=0$ outside $K$. Hence, $\phi=\psi \in A^q_\omega(\Gamma,\cX)$.

Now, if $\phi$ does not have compact support, then $\phi$ belongs to $A^q_\omega(\Gamma,\cX)$ locally at $\infty$ and so there is a function $g\in A^q_\omega(\Gamma,\cX)$ which coincides with $\phi$ outside some compact subset of $\Gamma$. Then $\phi - g$ has compact support and belongs to $A^q_\omega(\Gamma,\cX)$ locally at every point of $\Gamma$. So, by the first case, $\phi- g\in A^q_\omega(\Gamma,\cX)$ and $\phi \in A^q_\omega(\Gamma,\cX)$. This completes the proof.
\end{proof}

Now we turn to our main theorem in this section, which is for $G=\mathbb{R}^d=\Gamma$. First, we show a lemma which shows that $A^q_\omega(\mathbb{R}^d,\cX)$ has approximate identity.

\begin{proposition} \label{prop:suffR2}
Let $1\leq q <\infty$, $\omega$ be an admissible weight of regular growth (see Definition \ref{def:weight of regular growth}) which is a $q-$algebra weight if $q>1$, and let $f\in A^q_\omega(\mathbb{R}^d,\cX)$. For given $\epsilon >0$, there exists $\psi\in A^q_\omega(\mathbb{R}^d,\cX)$ with compact support such that $\| \psi f - f\|_{A^q_\omega(\mathbb{R}^d,\cX)} < \epsilon$. 
\end{proposition}
\begin{proof}
In view of Theorem \ref{Domar} and Remark \ref{Domarvector}, choose $\phi\in A^q_{\omega}(\mathbb{R}^d,\cX)$ such that $\phi(0)=1_\cX$ and support of $\phi$ is in the unit ball $B_1(0)=\{x\in\mathbb{R}^d:|x|<1\}$. For each $0<\lambda<1$, define 
\begin{align*}
    &\phi_\lambda(x)=\phi(\lambda x) \quad (x\in\mathbb{R}^d).
\end{align*}
Then $\phi_\lambda $ $ \in A^q_\omega(\mathbb{R}^d,\cX)$ as $\omega$ is of regular growth and its support is in $B_{\lambda^{-1}}(0)$. Using $1_\cX=\phi(0)= \int_{\mathbb{R}^d} \widehat{\phi}(y) dy$ and $(\widehat{\phi_\lambda}\ast\widehat{f})(\xi)=\int_{\mathbb{R}^d}\widehat{\phi}(y) \widehat{f}(\xi-\lambda y) dy$, we have
\begin{align*}
(\widehat{\phi_\lambda} \ast \widehat{f})(\xi) - \widehat{f}(\xi) = \int_{\mathbb{R}^d} \widehat{\phi}(y) (\widehat{f}(\xi-\lambda y) - \widehat{f}(\xi) ) dy.
\end{align*}
By Minkowski's inequality for integrals, we get
\begin{align}\label{eqn:5.1,1}
\|\phi_\lambda f-f\|_{A^q_\omega} \lesssim \int_{\mathbb{R}^d} \|\widehat{\phi}(y)\| \|\widehat{f}(\cdot-\lambda y) - \widehat{f}(\cdot) \|_{L^q_\omega} dy.
\end{align}
Note that for each $y\in\mathbb{R}^d$, 
\begin{align} \label{eqn:finiteintegral}
\nonumber
    \| \widehat{f}(\cdot-\lambda y) - \widehat{f}(\cdot) \|_{L^q_\omega} 
    &\leq \| \widehat{f}(\cdot-\lambda y)\|_{L^q_\omega} + \|\widehat{f}(\cdot) \|_{L^q_\omega} \\ \nonumber
    &= \left( \int_{\mathbb{R}^d} \|\widehat{f}(\xi-\lambda y) \|^q \omega(\xi-\lambda y+ \lambda y)^q d\xi \right)^\frac{1}{q} + \left( \int_{\mathbb{R}^d} \|\widehat{f}(\xi)\|^q \omega(\xi)^q d\xi \right)^\frac{1}{q} \\ \nonumber
    &\leq  \omega(\lambda y)\left( \int_{\mathbb{R}^d} \|\widehat{f}(\xi-\lambda y) \|^q \omega(\xi-\lambda y)^q d\xi \right)^\frac{1}{q} + \|\widehat{f}\|_{L^q_\omega} \\ 
    &\lesssim \omega(y) \|\widehat{f}\|_{L^q_\omega}
\end{align} as $\omega$ is a weight of regular growth and $0<\lambda<1$. And thus the integral in \eqref{eqn:5.1,1} is finite. Also, observe that $\displaystyle \lim_{\lambda \to 0} \left\|\phi_\lambda f -  f \right\|_{A^q_\omega} = 0$. So, choose $\lambda$ such that $\|\phi_\lambda f -  f \|_{A^q_\omega}<\epsilon$ and take $\psi=\phi_\lambda$.
\end{proof}

Moreover, we have the following lemma that can be seen as a corollary of the above result.

\begin{lemma} \label{lem:suffR1}
Let $1\leq q <\infty$, $\omega$ be an admissible weight of regular growth which is a $q-$algebra weight if $q>1$, and let $f,\phi\in A^q_\omega(\mathbb{R}^d,\cX)$. Then 
\begin{align}
    \left\| \widehat{\phi_\lambda} \ast \widehat{f} - \left( \int_{\mathbb{R}^d} \widehat{\phi}(y) dy \right) \widehat{f} \ \right\|_{L^q_\omega} \to 0 \quad \text{as} \quad \lambda\to0,
\end{align}
where $\phi_\lambda(x)=\phi(\lambda x)$ for $x\in\mathbb{R}^d$.
\end{lemma}

Now, we are ready for the proof of our main theorem.
\begin{proof}[\textbf{Proof of Theorem \ref{suff} \eqref{ii}}]
Let $f= f_1+ if_2 \in A^q_\omega(\mathbb{R}^d)$ be such that image of $f$ is in $I$, where $f_1$ and $f_2$ are real valued functions, and with an abuse of notation, we write $F(f) = F(f_1,f_2)$. To show that $F(f) \in A^q_\omega(\mathbb{R}^d,\cX)$, by Lemma \ref{lem:suff}, it is enough to show that $F(f)$ belongs to $A^q_\omega(\mathbb{R}^d)$ locally at each point of $\mathbb{R}^d$ and at infinity.

First we show that $F(f)$ belongs to $A^q_\omega(\mathbb{R}^d)$ locally at each point of $\mathbb{R}^d$. Fix $t_{0} \in \R^d $ and let $z_0=f(t_0)\in\mathbb{C}$. 
Since $F$ is real analytic at $z_0=x_0+iy_0$, there is some $\delta>0$ such that $F$ has series representation 
\begin{align}\label{eqn:power series of F}
F(z)=F(x,y)=F(z_0) + \sum_{(m,n)\in\mathbb{N}_0^2\setminus\{(0,0)\}} c_{mn} (x-x_0)^m (y-y_0)^n,
\end{align}
where $c_{mn}\in\cX$ and it absolutely converges in the norm of $\cX$ for all $z=x+iy$ with $|x-x_0|<\delta$ and $|y-y_0|<\delta$.
Then 
\begin{align}\label{eq:aq2}
F(f_1(t),f_2(t)) =  F(x_{0}, y_{0}) + \sum_{(m,n)\in\mathbb{N}_0^2\setminus\{(0,0)\}} c_{mn} [f_1(t)-f_1(t_{0})]^{m} [f_2(t)-f_2(t_{0})]^{n}
\end{align}
whenever the series converges. 
Note that both $f_1$ and $f_2$ are in $A^q_\omega(\mathbb{R}^d)$, being the real and imaginary part of $f$.

Let $\phi\in A^q_\omega(\mathbb{R}^d)$ be such that $\phi\equiv1$ in some neighborhood $V$ of $0$. For $\lambda>0$, define
\begin{align}\label{def:g_lambda}
g_\lambda(t)= \left[f\left(\lambda t + t_0\right) - f(t_0) \right] \phi(t) \quad (t\in\mathbb{R}^d).
\end{align}
Then $\widehat{g_\lambda}=\widehat{R_{t_0} f_\lambda} \ast \widehat{\phi} - (R_{t_0}f)(0) \widehat{\phi}$, where $(R_{t_0}f)(t)=f(t+t_0)$ and $f_\lambda(t)=f(\lambda t)$ for all $t\in \mathbb{R}^d$ and from Lemma \ref{lem:suffR1} it follows that that $g_\lambda\in A^q_\omega(\mathbb{R}^d)$ and $\|g_\lambda\|_{A^q_\omega}<\delta$ for some $\lambda$ near to $0$ and we now fix this $\lambda$.
Now, writing $g_\lambda=g_{\lambda,1} + i g_{\lambda,2}$, see that the function 
\begin{align}
h(t)=F(z_0) \phi(t) + \sum_{(m,n)\in\mathbb{N}_0^2\setminus\{(0,0)\}} c_{mn} g_{\lambda,1}(t)^m g_{\lambda,2}(t)^n
\end{align}
belongs to $A^q_\omega(\mathbb{R}^d,\cX)$ as the series on right side converges absolutely because $\|g_{\lambda,j}\|_{A^q_\omega}\leq\|g_\lambda\|_{A^q_\omega}<\delta$ for $j=1,2$.
So, if we set $\displaystyle g_{t_0}=h\left( \frac{t-t_0}{\lambda} \right)$. Then $g_{t_0}\in A^q_\omega(\mathbb{R}^d,\cX)$ and $g_{t_0}(t)=F(f(t))$ in $V_{t_0}=\{x+t_0:x\in V\}$ which is an open neighborhood of $t_0$. Thus, $F(f)$ belongs to $A^q_\omega(\mathbb{R}^d,\cX)$ locally at $t_0\in\mathbb{R}^d$. Since $t_0$ was chosen arbitrarily, $F(f)$ belongs to $A^q_\omega(\mathbb{R}^d,\cX)$ locally at each $t\in\mathbb{R}^d$.

To conclude we show that $F(f)$ belongs to $A^q_\omega(\mathbb{R}^d,\cX)$ locally at $\infty$. Taking $z_0=0$, that is, $(x_0,y_0)=(0,0)$ in equation \eqref{eqn:power series of F} and using the fact that $F(0)=0$,  the expansion \eqref{eq:aq2} becomes
\begin{align} \label{eq:aq1}
F(f_1(t),f_2(t)) = \sum_{(m,n)\in\mathbb{N}_0^2\setminus\{(0,0)\} } c_{mn} \, [f_1(t)] ^{m} \, [f_2(t)]^{n},
\end{align}  
whenever the series converges.

By Proposition \ref{prop:suffR2}, there is $\psi\in A^q_\omega(\mathbb{R}^d)$ such that $\psi\equiv1$ in some neighborhood $V$ of $0$, $\psi\equiv0$ outside compact set $K$ with $V\subset K$ and $\| (1-\psi) f_i\|_{A^q_\omega} = \| \psi f_i - f_i \|_{A^q_\omega} < \delta$ for $i=1,2$. Now consider the function $h$ defined by
\begin{align}
h(t)= \sum_{(m,n)\in\mathbb{N}_0^2\setminus\{(0,0)\} } c_{mn} \, [(1-\psi(t))f_1(t) ]^m \, [(1-\psi(t))f_2(t) ]^n.    
\end{align}
The above series is absolutely convergent in $A^q_\omega(\mathbb{R}^d,\cX)$, in view of the above norm estimates, hence $h \in A^q_\omega(\mathbb{R}^d,\cX)$. Moreover, the fact that $\psi$ is compactly supported in $K$ implies that $1-\psi \equiv 1 $ outside $K$ which gives $h=F(f)$ outside $K$. This shows that $F(f)$ belongs to $A^q_\omega(\mathbb{R}^d,\cX)$ locally at $\infty$. This completes the proof. 

\end{proof}

\subsection{Modulation and amalgam spaces}
Here we shall provide the proof of Theorem \ref{tfam} \eqref{sc}. First, we state some known results and provide some results required for proving it. We follow the outline of proof given in \cite{db} for $M^{p,1}$ and $W^{p,1}$.

Throughout this section $X$ is one of the space $M^{p,q}_{\omega}(\mathbb{R}^d)$, $W^{p,q}_\omega(\mathbb{R}^d)$ or $\widehat{w}^{p,q}_\omega(\mathbb{R}^d)$ for $1\leq p,q <\infty$ and $\omega$ is an admissible ($q-$algebra, if $q>1$) weight of regular growth on $\mathbb{R}^d$. 

First we provide an definition analogue to Def. \ref{def:belongslocally} for $X$.
\begin{definition}
Let $\phi$ be a function defined on $\mathbb R^{d}$. 
\begin{enumerate}
\item We say that $\phi$ belongs to $X$ locally at point $\gamma_{0}\in \mathbb R^{d}$ if there is a neighborhood $V$ of $\gamma_{0}$ and a function $h\in X$ such that $\phi(\gamma)= h(\gamma)$ for every $\gamma \in V.$
\item We say that $\phi$ belongs to $X$ locally at $\infty$ if there is a compact set $K\subset \mathbb R^{d}$ and  a function $h\in X$ such that $\phi(\gamma)=h(\gamma)$ in the complement of $K.$
\end{enumerate} 
\end{definition}
For convenience, denote the collection of functions belonging to $X$ locally at each $\gamma\in\mathbb{R}^d$ by $X_{loc}$, the collection of all functions belonging to $X$ locally at $\infty$ by $X_\infty$, and let $X_{loc,\infty} = X_{loc} \cap X_{\infty}$.

The next lemma gives the useful criterion for functions to be in $X$.

\begin{lemma} \label{lsuffma}
If $\phi$ belongs to $X$ locally at every point of  $\mathbb R^{d} \cup \{\infty \} $, then $\phi \in X.$
\end{lemma}
\begin{proof}
The proof goes on the same line as that of Lemma \ref{lem:suff} and uses Urysohn's lemma and the facts that $\mathbb{R}^d$ is a locally compact Hausdorff space and the algebra property of $X$ from Proposition \ref{proalgebra} \eqref{algebra}.
\end{proof}

The next proposition has been first established in \cite{db} for $M^{p,1}$ and $W^{p,1}$ and in \cite{HGWL1, HGWL2} for $M^{p,q}_{\omega_s}$, where $1\leq p \leq \infty$ and Japanese weight $\omega_s=(1+|x|^2)^\frac{s}{2}$ with $s=0$ for $q=1$ and $s>\frac{d}{q'}$ for $1<q<\infty$. We note that the same techniques can be applied for the proofs here for the weights of regular growth. And with this in mind, we avoid giving the proof in detail.

\begin{proposition}\label{kp} 
Let $ f\in X$, and let $\epsilon>0.$ There exists a $\psi \in X$ having compact support such that $\| (1-\psi) f \|_{X}<\epsilon$.
\end{proposition}
\begin{proof} 
The proof follows by taking $\phi\in X$ having compact support with $\phi(0)=1$, defining $\phi_\lambda(x)=\phi(\lambda x)$ for $x\in\mathbb{R}^d$, $0<\lambda\leq1$ and employing the techniques used in Proposition \ref{prop:suffR2} with minor modifications.
\end{proof}

The proof of Theorem \ref{tfam} \eqref{sc} follows using the lines of arguments presented in the proof of Theorem \ref{suff} \eqref{ii} and uses Proposition \ref{kp} in stead of Proposition \ref{prop:suffR2}.

\section{Appendix}\label{apd}

Even though we need only unweighted case, we prove Proposition \ref{B1}, in more generality, in setting of Japanese weight. It shall be noted that we do not need the algebra structure here. It follows from the following proposition.

\begin{proposition}\label{faT} Let $1\leq p,q < \infty$, $\omega_s(x)=\langle x \rangle^s=(1+|x|^2)^\frac{s}{2}$ $(x\in\mathbb{R}^d, s\geq0)$ be the Japanese weight on $\mathbb{R}^d$, $f$ be a function on $\mathbb R^d,$ and let $\phi$ be a smooth cut-off function supported on $\mathbb T^d=[0, 2\pi)^d$. Then
\begin{enumerate}
\item \label{1} $ \|\phi f\|_{A^{q}_{\omega_s}(\mathbb T^d)} \lesssim \|f\|_{\widehat{w}^{p,q}_{\omega_s}},$
\item \label{2} $\|\phi f\|_{\widehat{w}^{p,q}_{\omega_s}}  \lesssim  \|f\|_{A^{q}_{\omega_s}(\mathbb T^d)}.$
\end{enumerate}
\end{proposition}
\begin{proof} 
\eqref{1} Put $Q=[-1/2, 1/2)^d$. Then $\sum_{n \in \mathbb Z^d} \chi_{Q} (\xi-n)\equiv 1$ for all $\xi\in\mathbb{R}^d$. Let $g$ be a smooth cutoff function supported on $[-1,1]^d$ such that $g(\xi)=1$ for $\xi\in Q$. Then $\chi_{Q}(\xi) g(\xi)= \chi_{Q}(\xi)$ for all $\xi \in \mathbb R^d.$
For $n\in\mathbb{Z}^d$, we have
\begin{eqnarray*}
\widehat{\phi f}(n) 
& = &  \int_{\mathbb T^d } \phi (y) f(y) e^{-iny} dy\\
& = & \int_{\mathbb R^d}  \phi (y) f(y) e^{-iny} dy \\
& = & (\widehat{ \phi} \ast \widehat{f}) (n)\\
& = & \sum_{m \in \mathbb  Z^d} \int_{\mathbb R^d} \chi_{Q} (\xi-m) \widehat{f}(\xi) \overline{\widehat{\phi} (\xi-n)} d\xi\\
& = & \sum_{m \in \mathbb  Z^d} \int_{\mathbb R^d} \chi_{Q} (\xi-m) \widehat{f}(\xi) g(\xi-m)\overline{  \widehat{ \phi} (\xi-n)} d\xi.
\end{eqnarray*}
Note that $g, \widehat{\phi} \in \mathcal{S}(\mathbb R^d)$ and so  $|\widehat{\phi}(\xi)| \leq  \frac{C_s}{(1+ |\xi|^2)^{s/2}} $ for all $\xi \in \mathbb R^d$ and $s\geq 0.$
For $m\neq n \in \mathbb Z^d,$ we note that
\begin{eqnarray*}
\left\| g(\xi-m)\overline{  \widehat{ \phi} (\xi-n)} \right\|_{L^{p'}_{\xi}}^{p'} & = & \int_{|\xi-m|_{\infty}\leq 1} \left|g(\xi-m)\overline{  \widehat{ \phi} (\xi-n)}\right|^{p'} d\xi\\
&= & \int_{|\xi|_{\infty}\leq 1} \left|g(\xi)\overline{  \widehat{ \phi} (\xi +m-n)}\right|^{p'} d\xi\\
& \lesssim & \|g\|_{L^{\infty}(|\xi|_{\infty}\leq 1)}  \int_{|\xi|_{\infty}\leq 1} \left|\overline{  \widehat{ \phi} (\xi +m-n)}\right|^{p'} d\xi\\
& \lesssim  &  \frac{1}{\langle \xi + m-n  \rangle^{p'(s+2)}}  \quad \quad (|\xi|_{\infty} \leq 1)
\end{eqnarray*}
For  $|\xi|_{\infty} \leq 1$ and  $|n-m|_{\infty} \geq 2,$ we  have $|\xi| \leq |n-m|_{\infty}/2$ and $|n-m + \xi|_{\infty} \geq \frac{1}{2} |n-m|_{\infty}$ and hence $C \langle \xi + m-n  \rangle^{-(s+2)} \leq C4^{(s+2)/2} \langle m-n  \rangle^{-(s+2)}$.  For $m\neq  n$ and $|n-m|_{\infty}<2$ ,  we have $|m-n|_{\infty}=1$ and so 
$C \langle \xi + m-n  \rangle^{-(s+2)} \lesssim  C= 2^{(s+2)/2}C \langle m-n  \rangle^{-(s+2)}.$ 
Now,  it follows that 
\begin{equation}\label{ie}
\left\| g(\cdot-m)\overline{  \widehat{ \phi} (\cdot-n)} \right\|_{L^{p'}} \lesssim    \frac{1}{\langle  m-n  \rangle^{(s+2)}}. 
\end{equation}
By H\"older's inequality and \eqref{ie},  we have 
\begin{eqnarray}
|\widehat{\phi f}(n)|  & \lesssim   &  \sum_{m \in \Z^d}  \frac{1}{\langle  m-n  \rangle^{(s+2)}} \|\chi_{Q}(\cdot -m) \widehat{f}\|_{L^p}
\end{eqnarray}
By $\langle n \rangle^s\lesssim \langle n-m \rangle^s \langle m \rangle^s $ for $s\geq 0,$ \eqref{ie} and Young's  inequality we obtain
\begin{eqnarray*}
\|\phi f\|_{\mathcal{F}L^{q}_s(\mathbb T^d)} & = & \| \langle n \rangle^s \widehat{\phi f}(n)\|_{\ell^q} \lesssim \left\| \sum_{m \in \Z^d} \frac{1}{\langle n-m \rangle^2} \langle m \rangle^s \|\chi_{Q}(\cdot -m) \widehat{f}\|_{L^p}  \right\|_{\ell^q}\\
& \lesssim & \left\| \langle n \rangle^s \|\chi_{Q}(\cdot -n) \widehat{f}\|_{L^p}  \right\|_{\ell^q}= \| f\|_{\widehat{w}^{p,q}_s}   \\
\end{eqnarray*}

\noindent
\eqref{2} Suppose that $\psi$ is smooth cutoff  function supported on $[-1/2+\epsilon, 1/2+ \epsilon)^d$ for some $\epsilon>0.$  Let $n \in \mathbb Z^d.$ Then,  we have 
\begin{eqnarray*}
\mathcal{F}^{-1} (\psi (\xi -n) \widehat{\phi f} (\xi))(x) & = & \int_{\R^d} \psi (\xi-n) \widehat{\phi f} (\xi) e^{i\xi x} d\xi\\
& = & \int_{\mathbb T^d} \phi (y) f(y) \int_{\R^d} \psi (\xi -n) e^{i \xi (x-y)} d\xi dy\\
& = & e^{inx} \int_{\T^d} f(y) \phi (y) \psi^{\vee} (x-y) e^{-iny} dy\\
& =  & e^{inx} \int_{\T^d} \sum_{m \in \Z^d} \widehat{f}(n-m) e^{i (n-m)y} \phi(y) \psi^{\vee} (x-y) e^{-iny} dy\\
& = & e^{inx} \sum_{m \in \Z^d} \widehat{f}(n-m) F_m(x),
\end{eqnarray*}
where $\widehat{f}(n-m)$ are the Fourier coefficients of $f$ (as a function on $\T^d$) at $n-m$ and 
\[F_m(x)= \int_{\R^d} \phi (y) \psi^{\vee} (x-y) e^{-imy} dy. \]
Taking the Fourier transform both sides,  we obtain
\begin{eqnarray*}
\psi (\xi -n) \widehat{\phi f} (\xi)
& = & \int_{\R^d}e^{inx} \sum_{m \in \Z^d} \widehat{f}(n-m) F_m(x) e^{-ix\xi} dx\\
& = & \sum_{m\in \Z^d} \widehat{f}(n-m) \int_{\R^d} \int_{\R^d} \phi (y) \psi^{\vee} (x-y) e^{-imy} dy  \ e^{i (n-\xi)x} dx\\
& = & \sum_{m\in \Z^d} \widehat{f}(n-m)\phi (y) \int_{\R^d}  \int_{\R^d}  \psi^{\vee} (x-y) e^{i (n-\xi)x} dx  \  e^{-imy} dy  \\
& = & \sum_{m\in \Z^d} \widehat{f}(n-m) \int_{\R^d}\phi (y)  \psi (\xi -n) e^{i (n-\xi)y} \  e^{-imy} dy \\
& = & \sum_{m\in \Z^d} \widehat{f}(n-m)   \psi (\xi -n)  \widehat{\phi} (m-n+\xi)
\end{eqnarray*}
Note that 
\begin{eqnarray*}
\| \psi (\xi -n)  \widehat{\phi} (m-n+\xi)\|_{L^p_{\xi}} & = & \int_{\R^d} | \psi (\xi)  \widehat{\phi} (m+\xi)|^p d\xi\\
& \lesssim & \int_{|\xi|_{\infty} \leq 1+\epsilon} |\widehat{\phi} (m+\xi)|^p d\xi\\
& \lesssim &  \frac{1}{\langle m \rangle ^{s+2}}
\end{eqnarray*}
and 
\begin{eqnarray*}
\|\psi (\xi -n) \widehat{\phi f} (\xi)\|_{L^p_{\xi}} & \lesssim & \sum_{m\in \Z^d} \widehat{f}(n-m) \frac{1}{\langle m \rangle ^{s+2}}.
\end{eqnarray*}
By Young's inequality, we obtain
\begin{eqnarray*}
\|\phi f \|_{\widehat{w}^{p,q}_s} & = & \|\langle n \rangle^s \| \chi_{Q}(\xi -n) \widehat{f}\|_{L^p_\xi} \|_{\ell^q} \lesssim  \left\| \sum_{m \in Z^d} \langle n-m \rangle^s |\widehat{f} (n-m) |\frac{1}{ \langle m \rangle^2} \right\|_{\ell^q}\\
& \lesssim & \| \langle n \rangle^s \widehat{f}(n) \|_{\ell^q}= \|f\|_{\mathcal{F}L^q_s}.
\end{eqnarray*}
\end{proof}

\section*{Acknowledgments}
The second author gratefully acknowledges the Post Doctoral Fellowship under the ISIRD project 9--551/2023/IITRPR/10229 of Dr. Manmohan Vashisth from Indian Institute of Technology Ropar. The second author is also grateful to Indian Institute of Science Education and Research Pune for financial assistance during his visit, which initiated this work, and to Harish-Chandra Research Institute, Prayagraj (Allahabad) for the RA position under Dr. Saikatul Haque for 2 months, which partially supported this work. This work was partially supported by the FIST program of the Department of Science and Technology, Government of India, Reference No. SR/FST/MS--I/2018/22(C). The authors would like to thank  professor H.G. Feichtinger for  pointing out several  useful references related to weights.

\section*{Statements and Declarations}
\textbf{Competing Interests.} The authors have no relevant financial or non-financial interests to disclose.

\textbf{Author Contributions.} All authors contributed equally.

\textbf{Availability of data and material.} Not applicable.

\bibliography{ref}
\bibliographystyle{plain}

\end{document}